\documentclass[a4paper,11pt]{amsart}

\usepackage{amsfonts,latexsym,rawfonts,amsmath,amssymb,amsthm, mathrsfs, lscape}

\usepackage[all]{xy}
\usepackage[english]{babel}
\usepackage[utf8]{inputenc}
\usepackage[T1]{fontenc}
\usepackage{graphicx}
\usepackage{booktabs}
\usepackage{array, tabularx}
\usepackage{setspace}
\usepackage{textcomp}
\usepackage{tikz}
\usepackage{multirow}
\usepackage{paralist}
\usepackage{comment}

\usepackage[hypertexnames=false,
backref=page,
    pdftex,
    pdfpagemode=UseNone,
    breaklinks=true,
    extension=pdf,
    colorlinks=true,
    linkcolor=blue,
    citecolor=blue,
    urlcolor=blue,
]{hyperref}

\usepackage{amsrefs}

\newtheorem*{teo*}{Theorem}
\newtheorem*{lem*}{Lemma}
\newtheorem*{cor*}{Corollary}
\newtheorem*{pro*}{Proposition}
\newtheorem*{hyp*}{Hypothesis}
\newtheorem*{claim*}{Claim}
\newtheorem*{probl*}{Problem}
\newtheorem*{conj*}{Conjecture}
\newtheorem*{exe*}{Exercise}
\newtheorem*{exa*}{Example}
\newtheorem*{def*}{Definition}
\newtheorem*{rem*}{Remark}
\newtheorem{question}{Question}
\newtheorem{teo}{Theorem}[section]
\newtheorem{cor}[teo]{Corollary}
\newtheorem{lem}[teo]{Lemma}

\newtheorem{pro}[teo]{Proposition}
\newtheorem{rem}[teo]{Remark}

\theoremstyle{definition}
\newtheorem{exa}[teo]{Example}
\newtheorem{defi}[teo]{Definition}

\newcommand{\C}{\mathbb{C}}
\newcommand{\N}{\mathbb{N}}
\newcommand{\Q}{\mathbb{Q}}
\newcommand{\R}{\mathbb{R}}
\newcommand{\T}{\mathbb{T}}
\newcommand{\Z}{\mathbb{Z}}
\newcommand{\Pro}{\mathbb{CP}}
\newcommand{\ccal}{\mathcal{C}}
\newcommand{\fcal}{\mathcal{F}}
\newcommand{\hcal}{\mathcal{H}}
\newcommand{\mcal}{\mathcal{M}}
\newcommand{\ocal}{\mathcal{O}}
\newcommand{\pcal}{\mathcal{P}}
\newcommand{\ucal}{\mathcal{U}}
\newcommand{\zcal}{\mathcal{Z}}
\newcommand{\nsf}{{\sf N}}
\newcommand{\fsf}{{\sf F}}
\newcommand{\lsf}{{\sf L}}
\newcommand{\de}{\partial}
\newcommand{\debar}{\overline{\partial}}

\newcounter{tempo}
\title[Minimal kernels and compact analytic objects in complex surfaces]{Minimal kernels and compact analytic objects in complex surfaces}

\author[S.~Mongodi]{Samuele Mongodi\textsuperscript{1}}
\address[\textsuperscript{1}]{Dipartimento di Matematica - Politecnico di Milano, Via Bonardi 9, I--20133, Milano,  Italy}
\email{samuele.mongodi@polimi.it}
\author[G.~Tomassini]{Giuseppe Tomassini\textsuperscript{2}}
\address[\textsuperscript{2}]{Scuola Normale Superiore, Piazza dei Cavalieri, 7 - I-56126 Pisa, Italy}
\email{giuseppe.tomassini@sns.it}

\date{\today}

\begin{document}
%\nocite{*}

\begin{abstract}
 In this paper, we want to study the link between the presence of compact objects with some analytic structure and the global geometry of a weakly complete surface. We begin with a brief survey of some now classic results on the local geometry around a (complex) curve, which depends on the sign of its self-intersection and, in the flat case, on some more refined invariants (see the works of Grauert, Suzuki, Ueda). Then, we recall some results about the propagation of compact curves and the existence of holomorphic functions (from the works of Nishino and Ohsawa). With such considerations in mind, we give an overview of the classification results for weakly complete surfaces that we obtained in two joint papers with Slodkowski (see \cite{Mon-Slo-Tom}, \cite{Mon-Slo-Tom1} and we present some new results which stem from this somehow more local (or less global) viewpoint (see Sections \ref{class}, \ref{annuli} and \ref{struct}).
\end{abstract}

\maketitle
\tableofcontents

\section*{Introduction}

The Levi problem, in its broadest formulation, asks for geometric conditions to guarantee that a given complex space is (a modification of) a Stein space or, from another point of view, which geometric obstruction are there to the existence of ``many'' holomorphic functions.

%depending on the precise meaning that we give to the word ``many'', we obtain, among the others, holomorphically convex spaces, modifications of Stein spaces or Stein spaces.

\medskip

The geometric characteristics we look for are usually encoded in the existence of some particular defining function (for domains contained in an ambient space) or of some particular exhaustion function. As it is well known, the class of functions that turns out to work well with holomorphic functions is the class of ({\em strictly}) {\em plurisubharmonic functions}.
The original formulation of the Levi problem, i.e. that every domain in $\C^n$ with smooth pseudoconvex boundary is a domain of holomorphy, was solved by Oka \cites{Oka1, Oka2}, Bremermann \cite{Bre}, Norguet \cite{Nor}. A few years later, Grauert tackled and proved the generalization of the Levi problem to a complex manifold, proving that a complex manifold is Stein if and only if it admits a smooth strictly plurisubharmonic exhaustion \cite{Gra2}, and Narasimhan generalized the result to complex spaces \cites{Nar2-I, Nar2-II}, allowing the exhaustion function to be just continuous.

\medskip

A natural question is to ask what happens if we allow the exhaustion to be only plurisubharmonic, i.e. if we allow its Levi-form to degenerate somewhere. The resulting spaces are called \emph{weakly complete}. Quite obviously, the class of weakly complete spaces includes Stein spaces, but it is not limited to them, as any space which is proper over (i.e. admits a holomorphic proper surjective map onto) a Stein space is weakly complete. Grauert produced an example of a weakly complete space whose only holomorphic functions are the constants, thus showing that not all weakly complete spaces are holomorphically convex (see \cite{Nar}). As a partial converse, a result by Ohsawa, generalized to complex spaces by Vajaitu and the second author (see \cites{Ohsawa, Tomassini-VA} and Section \ref{Ohsawa}), shows that, in dimension $2$, a weakly complete (complex) surface is holomorphically convex if and only if it admits a nonconstant holomorphic function.

\medskip

At first sight, it seems that under the hypothesis of weakly completeness there is no way to tell ``how many'' holomorphic functions exist on our space. However, a closer look at some examples (see Section \ref{examples}) reveals that we obtain more precise information as soon as we look at \emph{how far} our space $X$ is from being Stein: for example, if we have a proper surjective holomorphic map $p:X\to Y$ with positive-dimensional fibers onto a Stein space $Y$, then every plurisubharmonic function on $X$ will be constant on each fiber $p^{-1}(y)$ for $y\in Y$ and every holomorphic function on $Y$ will give, by pullback, a holomorphic function on $X$; on the other hand, in Grauert's example, $X$ contains Levi-flat hypersurfaces whose Levi foliation has dense leaves, therefore $\ocal(X)=\C$.

\medskip

In order to study the obstructions that prevent the existence of a strictly plurisubharmonic function on a complex manifold (or a complex space) $X$, Slodkowski and the second author introduced in \cite{Slo-Tom} the concept of the minimal kernel $\Sigma_X$, as the set of points where no exhaustion function for $X$ can be strictly plurisubharmonic (see Section \ref{prelim}). The crucial property of $\Sigma_X$ is the following: the intersections of $\Sigma_X$ with the regular level sets of a plurisubharmonic exhaustion function enjoy the so called ``local maximum property'' and, in dimension $2$, they are locally a union of complex discs. These compact sets obtained by slicing $\Sigma_X$ with the regular level sets of a plurisubharmonic exhaustion function and their complex structure play a fundamental role in the classification theorem of weakly complete surfaces proved in \cite{Mon-Slo-Tom} (see also Section \ref{class}). The other key ingredient is the presence of a real-analytic plurisubharmonic exhaustion function, which allows us to create a bridge between the local geometry, the geometry of a level set and the geometry of the whole surface.

In the quoted paper \cite{Mon-Slo-Tom}, two kinds of \emph{compact objects with an analytic structure} appeared, playing a significant role:
\begin{itemize}
\item compact complex curves, embedded in $X$,
\item immersed complex curves, with compact closure in $X$.\end{itemize}
They both force the plurisubharmonic functions to have degenerate Levi-form, but they behave quite differently in terms of holomorphic functions. Also, their contributions to the global geometry of the surface do not immediately seem equivalent. In particular, by a result of Nishino (see \cite{Nishino} and Section \ref{Nino}), the presence of a ``generic enough'' compact curve forces the whole weakly complete surface to be a union of compact curves, hence holomorphically convex with Remmert reduction of dimension $1$. On the other hand, immersed curves with compact closure can easily force any holomorphic function to be constant, as soon as the closure is ``large enough'' (e.g. contains a $3$-dimensional stratum, as a real-analytic set, which is the case in the presence of a real-analytic plurisubharmonic exhaustion function), but they do not seem to necessarily ``propagate'' to the whole surface, without some additional hypothesis. When a real-analytic plurisubharmonic exhaustion exists, the presence of an immersed curve with compact closure forces the whole surface (possibly outside an analytic set) to be foliated in $3$-dimensional Levi-flat hypersurfaces, in turn foliated with dense complex leaves. We called such surfaces ``of Grauert type''.

\medskip

The aim of this paper is twofold.

\smallskip

On one side, we give an account of what is known about the geometry of weakly complete surfaces, through the study of the compact sets we mentioned above. We start by recalling the results by Grauert \cite{Gra} and Suzuki \cite{Suzuki} on the neighborhood of a ``negative'' curve (i.e with self-intersection $(C^2)<0$) or ``positive'' curve (i.e with self-intersection $(C^2)>0$) respectively and move on to describe the classification of curves with zero self-intersection obtained by Ueda \cite{Ueda}. Then, we study under which conditions a compact curve propagates to generate a family of compact curves, by recalling the works by \cite{Nishino} and \cite{Ohsawa}. The link between the presence of compact curves and the global geometry is made explicit, in terms of the classification results proved in \cite{Mon-Slo-Tom}. We also briefly recall the results about the geometric structure of the Grauert-type surfaces, proved in \cite{Mon-Slo-Tom1}. Finally, using the results of \cite{Mon-Slo-Tom} we give a slightly different proof of a statement by Brunella \cite{Brunella} on the non-existence of a real-analytic plurisubharmonic exhaustion function on some classes of surfaces.

\smallskip

On the other side, we present some new results which stem from this somehow more local (or less global) viewpoint.

We show that the classification result from \cite{Mon-Slo-Tom} essentially holds when the hypotheses are true outside of a compact (see Theorem \ref{teo_buco} and Corollary \ref{cor12}).

We give a classification theorem for coronae (see Section \ref{annuli} for the precise definition), obtaining, also here, three cases (see Theorem \ref{teo34}): Grauert-type coronae, coronae which are proper over an open complex curve, coronae which are an increasing union of subcoronae with one strictly pseudoconvex boundary component. We show, by an example of Rossi \cite{Rossi}, that the last case is not always the complement of a compact set in a (modification of a) Stein space.

We prove some results for weakly complete surfaces with a plurisubharmonic exhaustion function which is supposed to be only smooth, in the hypothesis that the minimal kernel coincides with the whole surfaces. In this case, our surface is either proper over an open complex curve or the regular level sets of the exhaustion are Levi flat and foliated by dense complex curves, as in Grauert-type surfaces (Theorem \ref{teo_smooth}).

Finally, we introduce some generalizations of the notion of minimal kernel, due to Slodkowski \cite{Slo}, and employ one of them to give a condition under which the Levi problem for a pseudoconvex domain in a general complex manifold has a positive answer (see Theorem \ref{teo_levi}).

\medskip

The content of the paper is organized in 4 sections. In Section 1, we recall the basics on the minimal kernel and its slices along the level sets of a plurisubharmonic exhaustion. Section 2 is devoted to the study of the nature and the propagation of compact curves in a complex surface. In Section 3, we collect some classification results, known and new, and some examples. Finally, Section 4 contains the results on the smooth case, the description of the generalizations of the minimal kernel and their application to the Levi problem.

\section{The minimal kernel and its slices}\label{prelim}

A (reduced, connected) complex space $X$ is said {\em weakly complete} if there exists a (smooth, continuous) plurisubharmonic exhaustion function $\phi:X\to \R$ which is plurisubharmonic; obviously, a compact space statisfies this request in a trivial way, so we will further assume that $X$ is noncompact.

We know, from the results of Grauert and Narasimhan, that, when the exhaustion function can be taken to be everywhere strictly plurisubharmonic, the space is Stein; in general, it is a famous and hard problem to understand under which weaker conditions a space is a Stein space or a modification of a Stein space. Therefore, we study the obstructions that force a plurisubharmonic function to have a degenerate Levi-form.

\begin{defi}The $\ccal^k$ minimal kernel of $X$ is the set of points $x\in X$ such that every $\ccal^k$ plurisubharmonic exhaustion function fails to be strictly plurisubharmonic at $x$ and it is denoted by $\Sigma_X^{k}$. A \emph{$\ccal^k$ minimal function} is a function $\phi\in\ccal^k(X)$ which is plurisubharmonic on $X$ and strictly so exactly on $X\setminus\Sigma_X$.\end{defi}

We will usually employ the $\ccal^\infty$ minimal kernel and we will denote it simply by $\Sigma_X$; the definition and the main properties of the minimal kernel of a weakly complete space were first given in \cite{Slo-Tom}. We recall that a set $Y\subset X$ is called a \emph{local maximum set} if for every $x\in Y$ there exists a neighborhood $U$ of $x$ with the following property: for every compact $K\subset U$ and every (smooth) plurisubharmonic function $\psi$ defined in a neighborhood of $K$,
$$\max_{Y\cap K}\psi =\max_{Y\cap bK}\psi\;,$$
where the maximum over an empty set is understood to be $-\infty$. We will say that a set is a local maximum set or that it has the local maximum property.

Two significant results regarding the minimal kernel of a weakly complete space are the following.

\begin{pro}Let $u:X\to\R$ be a smooth plurisubharmonic exhaustion function; then for every $c\in\R$
$$\Sigma_X\cap\big\{x\in X\ :\ u(x)=c\big\}$$
is either empty or has the local maximum property.\end{pro}
See \cite{Slo-Tom}*{Theorem 3.6} for the case when $u$ is a minimal function, \cite{Mon-Slo-Tom}*{Theorem 3.2} for the general case, \cite{Slo}*{Lemma 4.8} for a further generalization.

\begin{pro}Let $X$ be a complex surface, $u:X\to\R$ be a smooth plurisubharmonic exhaustion function and $Y$ be a connected component of the level set $\{x\in X\ :\ u(x)=c\}$ such that $Y\cap\Sigma_X\neq\emptyset$. Then, for every point $p\in Y_{\textrm{reg}}\cap\Sigma_X$, there exist an open neighborhood $U\subset X$ and local coordinates $z,\ w$ on $U$ such that $U\cong \Delta_z\times\Delta_w$ and
$$U\cap Y\cap \Sigma_X\cong \bigcup_{t\in T}\big\{(z,f_t(z))\ :\ z\in \Delta_z\big\}$$
where each $f_t:\Delta_z\to\Delta_w$ is a holomorphic function.\end{pro}
See \cite{Slo-Tom}*{Lemma 4.1} for the case when $u$ is a minimal function, \cite{Mon-Slo-Tom}*{Proposition 3.5} for the general case.

\medskip

From these two results, there seems to be a strong relation between the presence of compact subspaces and the minimal kernel, at least in dimension $2$.

In general, one can say that a compact subspace or an immersed complex space with compact closure belong for sure to the minimal kernel; we do not know examples of spaces whose minimal kernel has a connected component which is not a union of the former.

\medskip

One striking property of the minimal kernel is the following propagation result  (see \cite{Slo-Tom}*{Theorem 3.9}).

\begin{teo}Let $X$ be a weakly complete manifold of dimension $\geq 2$ and $\phi:X\to\R$ be a $\ccal^2$ plurisubharmonic exhaustion function. Let $r>\min\phi$ and let $Y$ be a connected component of $\{x\in X\ :\ \phi(x)=r\}$, relatively open in the latter and that does not contain local minimum points of $\phi$.

If $Y$ is a local maximum set, there exists $s<r$ such that the topological boundary of the connected component $K$  of $\{x\in X\ :\ s<\phi(x)<r\}$ containing $Y$ is contained in $Y\cup\{x\in X\ :\ \phi(x)=s\}$. Then
\begin{itemize}
\item[$(a)$]$K$ is a connected compact set with nonempty interior,
\item[$(b)$]the forms
$$(\de\debar\phi)^{n-1}\wedge\de\phi\wedge\debar\phi,\ \ (\de\debar\phi)^{n-1}\wedge\de\phi,\ \ (\de\debar\phi)^{n-1}\wedge\debar\phi$$
vanish on $K$ and $(\de\debar\phi)^n$ vanishes on $K\setminus Y$,
\item[$(c)$]every level set
$$
\big\{x\in K\ :\ \phi(x)=t\big\}\;,
$$ for $s\leq t\leq r$, has the local maximum property.
\end{itemize}
\end{teo}

An immediate consequence on $\Sigma_X$ is the following, obtained by assuming that the level set considered is regular.

\begin{cor}Let $X$ be a weakly complete manifold of dimension $\geq 2$ and $\phi:X\to\R$ be a $C^2$ plurisubharmonic exhaustion function. If there is a regular value $r$ such that $\Sigma_X$ contains a connected component of the corresponding level set, then there is $s<r$ such that $\Sigma_X$ contains a connected component of the set $\big\{x\in X\ :\ s\leq \phi(x)\leq r\big\}$.\end{cor}

Gaining an understanding of the geometry of $\Sigma_X$ would give us some insight on the kind of obstructions that prevent a weakly complete space from being (a modification of) a Stein space.

\section{Compact complex curves.}\label{comphol}

As we saw, in complex surfaces, the obstruction to having a strictly plurisubharmonic exhaustion is linked to the presence of complex curves (embedded compact or immersed with compact closure). In this section, we want to collect and comment some results about the presence of compact curves in complex surfaces; there are two types of results we are interested in
\begin{itemize}
\item how the presence of a compact curve affects, in a neighborhood, the behavior of holomorphic and plurisubharmonic functions
\item which conditions guarantee the ``propagation'' of a compact curve.\end{itemize}
In other words, what properties ensure that a compact curve cannot be contained in a weakly complete surface? What does the presence of the curve tell us about the surface? When does the curve belong to a family of complex curves, e.g. the fibers of a proper map?

\subsection{The neighborhood of a compact curve}\label{nbd}
\subsubsection{Grauert Criterium} The first relevant result, in chronological order, is Grauert's investigation of \emph{exceptional curves} or, more generally, \emph{exceptional subspaces} of a complex space $X$.

Recall that a compact complex subspace $Y$ of $X$ is said to be {\em exceptional} if there is a proper holomorphic map $p:X\to X_0$ on a complex space $X_0$, such that $p(Y)$ is a point $x_0$ and $p_{\vert X\smallsetminus\, Y}:X\smallsetminus Y\to X_0\smallsetminus\{x_0\}$ is an isomorphism. By definition, exceptional subspaces are isolated.
%$$
%{\begin{split}
%\xymatrix{A\ar[r]^f\ar[d]^h& X\ar[d]^p\\
%B\ar[r]^g & Y}
%\end{split}}
%$$

A holomorphic vector bundle  $\fsf\to X$ is said to be {\em negative} if the zero section ${\bf 0}_{\fsf}$ of $\fsf$ has a strongly pseudoconvex neighborhood $U$, in particular $U$ is holomorphically convex. By Cartan theorem on quotient spaces \cite{Cartan} a neighborhod $U'\Subset U$ of ${\bf 0}_{\fsf}$ is proper on a Stein space $U_0$ and ${\bf 0}_{\fsf}$ is exceptional in $U$.

We have the following (see \cite{Gra}*{Satz 8})
\begin{teo*}[Grauert criterion] If the normal bundle $\nsf_{Y/X}$ is negative (i.e. $Y$ as 0-section of $\nsf_{Y/X}$ is exceptional) then $Y$ is exceptional in $X$.
\end{teo*}

If $X$ is a non-singular complex surface and $C$ is a compact, non-singular complex curve then $\nsf_{C/X}$ is negative if and only if the self-intersection $(C ^2)$ of $C$ is negative.

From Grauert criterium we deduce, in particular, that if $(C ^2)<0$ no compact (possibly singular) complex curve can be present near $C$.

The remaining cases $(C^2)>0$, $(C^2)=0$ were studied by Suzuki (see \cite{Suzuki}) and Ueda (see \cite{Ueda}) respectively.

We say that $C$ is {\em negative}, {\em positive}, {\em flat} if $(C^2)<0$, $(C^2)>0$, $(C^2)=0$, respectively.

\subsubsection{Positive curves}\label{zuki} The case of positive curves was studied by Suzuki; he obtained, as it may be expected, the opposite of Grauert's result (see \cite{Suzuki}*{Proposition 2.2}).
\begin{teo*}[Suzuki] If $X$ is a non singular complex surface and $C\Subset X$ is a compact non singular holomorphic curve such that $(C^2)>0$ then $C$ has a fundamental system of strongly pseudoconcave neighborhoods. In particular, $\ocal(X)=\C$.
\end{teo*}

In this case many compact holomorphic curves can be present in an arbitrary neighborhood of $C$ (e.g. $\Pro^1\subset\Pro^2$); however, no nonconstant holomorphic or plurisubharmonic function exists in a neighborhood of $C$, by pseudoconcavity.

\subsubsection{Ueda paper}\label{Ue} In \cite{Ueda} Ueda considers compact non-singular complex curves $C$ which are flat, i.e. such that $\nsf_{C/X}$ is topologically trivial. It is well known that $\nsf_{C/X}$ is represented by an element of $H^1(C,\mathbb S^1)$, which then coincides with the {\em Jacobian variety} ${\rm Pic^0}(C)$.

The paper of Ueda is very deep. In order to state the main results of  \cite{Ueda}) we need some preliminary definitions

\begin{defi}A line bundle $\lsf\to C$ is said to be of {\em finite order $m$} if $\lsf^{-m}$ is holomorphically trivial  and $m$ is the minimum with this property. If no such $m$ exists,  $\lsf$ is said to be of {\em infinite order}. The {\em order} of $C$ is the order of its normal bundle $\nsf_{C/X}$. It is denoted by {\rm ord}(C) and we have $1\le {\rm ord}(C)\le+\infty$.\end{defi}

The next fundamental concept is the {\em type} of the curve $C$.

\medskip

Let $\{V_j\}_{1\le j \le m}$ be a covering of $C$
by bidiscs $V_j=\{\vert z_j\vert<1, \vert w_{j}\vert<1\}$,
where $z_j$ and $w_j$ are local holomorphic coordinates and $C_j:=C\cap V_j=\{w_j=0\}$.
Let $U_j=V_j\cap C$ and $\ucal=\{U_j\}_{1\le j \le m}$. Then the coordinates $w_{(j)}$ can be chosen in such a way that $\nsf_{C/X}$ can be represented by a cocycle
\begin{equation}\label{cocy}
t_{jk}=\big(w_j/w_k\big)_{\vert U_j\cap U_k}\in {\rm Z}^1(\ucal,\mathbb S)
\end{equation}
Let  $\{t_{jk}\}$, $\{w_j\}$ be fixed; then $w_j-t_{jk}w_k$ is vanishing on $U_j\cap U_k$ so $w_j-t_{jk}w_k=f_{jk}(z_j)w_j^{\nu+1}$ with $f_{jk}(z_j)\neq 0$ holomorphic. The system $\{w_j\}$ is then said of {\em type $\nu$}. One check easily that $\{f_{jk}\}$ is a cocycle with values in $\nsf_{C/X}^{-\nu}$: $\{f_{jk}\}\in {\rm Z}^1(\ucal,\nsf_{C/X}^{-\nu})$.

The cocycle $\{f_{jk}\}$ is called the $\nu^{\rm th}$ {\em obstruction} associated to the system $\{w_j\}$.

\begin{defi}The curve $C$ is said of {\em finite type $n$} ({\rm type}(C)=n) if there exists a system $\{w_j\}$ of type $n$ such that $n^{\rm th}$ obstruction associated to $\{w_j\}$ is not cohomologous to zero. The curve $C$ is said of {\em infinite type} (${\rm type}(C)=+\infty)$ if the obstruction associated to every system $\{w_j\}$ is cohomologous to zero.\end{defi}

The idea under the previous computations is to measure the degree of coincidence of the extension of $\nsf_{C/X}$ to a neighborhood of $C$ as a flat bundle and the line bundle $[C]$, corresponding to the divisor of $C$.

It is a simple matter to prove that the previous definition is well posed; namely, that

\begin{itemize}
\item[1)] If there exists a system of type $n$ whose $n^{\rm th}$ obstruction is not cohomologous to zero, then no system $\{w_j\}$ exists of type $\nu>n$.
\item[2)] If there exists a system of type $n$ whose $n^{\rm th}$ obstruction is not cohomologous to zero, for every system of type $\nu<n$ the $\nu^{\rm th}$ obstruction is cohomologous to $0$.
\item[3)] If the obstruction associated to every system $\{w_j\}$ is cohomologous to zero, then for every $\nu$ there exists a system of type $\nu$.
\item[4)] The type of a curve does depend neither on the covering $\ucal$ nor on the cocycle $\{t_{jk}\}$.
\end{itemize}

\begin{rem*}If the order is infinite, the type is infinite as well.\end{rem*}

The type and the order help us divide compact non singular complex curves into four classes:
\begin{itemize}
\item if ${\rm ord}(C)$ is finite and ${\rm type}(C)$ is finite, we say that $C$ belongs to Class $\alpha$
\item if ${\rm ord}(C)$ is finite, but ${\rm type}(C)$ is infinite, we say that $C$ belongs to Class $\beta'$
\item if ${\rm ord}(C)$ is infinite, then also ${\rm type}(C)$ is infinite and we identify two classes
\begin{itemize}
\item if there exist a neighborhood $V$ of $C$ and a holomorphic function $u$ on a covering manifold of $V$ such that $|u|$ defines a single-valued function on $V$ and the divisor of $u$ is $C$, then we say that $C$ belongs to Class $\beta''$
\item otherwise, we say that $C$ belongs to Class $\gamma$.
\end{itemize}
\end{itemize}

\medskip

The situation near a curve $C$ in Class $\alpha$ is summarized by the following results.
\begin{itemize}
\item [($\alpha_1$)]For every $n'>{\rm type}(C)$ there exists a neighborhood $V_0$ of $C$ and a strongly plurisubharmonic function $\Phi:V_0\smallsetminus C\to\R$ such that $\Phi(p)\sim {\rm dist}(p)^{-n'}$ as $p\to C$. It follows that, for $c\to+\infty$ the sets $\{\phi>c\}\cup C$ give a fundamental system of strongly pseudoconcave neighborhoods of $C$. In particular $\ocal(X)=\C$.
\item [($\alpha_2$)] Let $V$ an open (connected) neighborhood of $C$ and $\Psi:V\smallsetminus C\to\R$ a strongly plurisubharmonic function such that $\Psi(p)=o({\rm dist}(p)^{-n''})$, $0<n''<n$, as $p\to C$. Then $\Psi$ is constant near $C$. In particular, if $f\in\ocal(V\smallsetminus C)$ and $\log^+\vert f(p)\vert=o({\rm dist}(p)^{-n''})$ then $f$ is constant.
\end{itemize}

For the curves in Class $\beta'$, we have the following.
\begin{itemize}
\item [($\beta'$)] If ${\rm ord}(C)=m$, on a neighborhood $V$ of $C$ there is a multivalued holomorphic function $u$ such that $u^m\in\ocal(V)$ and whose divisor is $m\,C$.
\end{itemize}
Conversely, it is easy to show that if  there is a non constant holomorhic function on a neighborhood of $C$ then $C$ belongs to Class $\beta'$.

\smallskip

Finally, by definition, if $C$ belongs to Class $\beta''$
\begin{itemize}
\item [($\beta''$)] on a neighborhood $V$ of $C$ there is a multivalued holomorphic function $u$ such that $|u|$ is single-valued on $V$ and whose divisor is $m\,C$.
\end{itemize}

\medskip

We analyze more closely what happens near a curve of Class $\beta'$ or Class $\beta''$.

\smallskip

If $C\in {\rm Class}\,\beta'\cup{\rm Class}\,\beta''$ then there exists a multivalued function  holomorphic in a neighborhood $V$ of $C$. Let $\epsilon>0$ such that $V_\epsilon:=\big\{p\in V:\vert u(p)\vert>\epsilon\big\}\Subset V$. For all $r$ with $0<r<\epsilon$, the boundary ${\rm b}V_\epsilon$ of $V_\epsilon$ is defined by the pluriharmonic function $\log\vert u\vert-\log\epsilon=0$, so it is Levi flat. The neighborhoods $V_\epsilon{'s}$ are said {\em pseudoflat}.

Observe that, due to the presence of the non constant pluriharmonic function $\log\vert u\vert$ no neighborhood $W\Subset V$ can be strictly pseudoconcave. Moreover, if $C$ belongs to Class $\beta'$ for $\vert c\vert\le\epsilon^m$ the curves $u^m={\rm const}$ are compact and irreducible.

Suppose now that $C$ is in  Class $\beta''$ and let $\Sigma_r:=\{p\in V:\vert u(p)\vert>r\}\Subset V$, $0<r<\epsilon$. $\Sigma_r$ is Levi flat hence foliated by holomorphic curves. Then (see \cite{Ueda}*{Sections 2 and 3}) $u$ defines a holomorphic foliation $\fcal$ on $V$ such that every leaf of $\fcal$, except for $C$, is contained in some $\Sigma_r$ and dense in it. It follows that every plurisubharmonic function on a neighborhood of a $\Sigma_r$ is constant on $\Sigma_r$. Consequently, $\ocal(V)=\C$.

\subsubsection{Curves near $C$}\label{nbhcurve} Saving the same notations, let $V$ be a fixed tubular neighborhood of $C$ and $\Gamma$ a $2$-cycle. Since $\nsf_{C/X}$ is topologically trivial $\Gamma\sim mC$ for some $m\in\Z$, hence $(\Gamma,C)=0$. In particular, if $\Gamma\neq C$ is a compact, irreducible complex curve then $\Gamma\cap C=\emptyset$ i.e. $\Gamma\Subset V\smallsetminus C$.
\begin{itemize}
\item[i)] in view of $(\alpha_1)$, $C$ does not belong to  Class $\alpha$
\item[ii)] if $C$ belongs to Class $\beta'$ then by $(\beta'_1)$ the compact curves $u^m={\rm const}$ are the only ones belonging to $V\smallsetminus C$
\item[iii)] by similar arguments one proves that if $C$ is in Class $\beta''$ then in a neighborhood of $C$ there is no compact complex curve other than $C$.
\end{itemize}
The situation concerning the curves belonging to Class $\gamma$ is rather mysterious (see \cite{Ueda}).

\begin{rem} There are at most countably many  curves of Class $\alpha$ or Class $\beta''$; indeed, every such curve has an open neighborhood where no other curve is present, hence there can be at most countably may.
\end{rem}

\subsection{Propagation of compact curves}\label{propag}
 \subsubsection{Nishino's paper}\label{Nino}
In the paper \cite{Nishino} the existence of a non constant holomorphic function $f:X\to\C$ (and more generally of holomorphic maps $X\to R$ where R is Riemann surface) is proved under the condition that $X$ contains at least a {\em generic cuve}. Let us recall some preliminary results proved in \cite{Nishino}.

\medskip
We start from the definition of a generic curve: morally, a curve $S$ is called generic if it is locally ``movable'', i.e. if there are a neighborhood $U$ of $S$ and a holomorphic function $F\in\ocal(U)$ such that $\{x\in U\ :\ F(x)=0\}=S$. However, the original definition by Nishino is more involved.

Let $S\Subset X$ be a compact, non singular complex curve, $\{V_j\}_{1\le j \le m}$  a covering of $S$ by bidiscs $V_j=\{\vert z_{(j)}\vert<1, \vert w_{(j)}\vert<1\}$, where $z_{(j)}$, $w_{(j)}$ are local holomorphic coordinates and $S_j:=S\cap V_j=\{w_{(j)}=0\}$. Such a covering is called {\em canonical}.

\smallskip

Let ${\rm(}\pcal{\rm)}=\{f_j\}_{1\le m}$, $f_j:V_j\to\C$ meromorphic, be a datum for the additive Cousin problem on  $V:=\cup_{j=1}^mV_j$, so $f_{ij}=f_i-f_j\in\ocal(V_i\cap V_j)$ for all $i,\ j$, and let ${\rm(}\zcal{\rm)}=\{g_j\}_{1\le m}$, $g_j:V_j\to\C$ holomorphic, be a datum for the multiplicative Cousin problem on  $V$, so $g_{ij}=g_i/g_j\in\ocal(V_i\cap V_j)$ for all $i,\ j$.
\begin{itemize}
\item ${\rm(}\pcal{\rm)}$ is said to be {\em solvable} on $X$ if there exist holomorphic functions $\phi_j\in \ocal(S_j)$  such that ${f_{jk}}_{\vert S_j\cap S_k}=\phi_j-\phi_k$ for all $1\le j,k\le m$.
\item ${\rm(}\zcal{\rm)}$ is said to be {\em solvable} on $X$ if there exist nowhere vanishing holomorphic functions $\psi_j\in \ocal^*(S_j)$ such that ${g_{jk}}_{\vert S_j\cap S_k}=\psi_j/\phi_k$.
\end{itemize}

\begin{defi}The complex curve $S$ is said {\em generic} if
\begin{itemize}
\item[a)] every datum ${\rm(}\pcal{\rm)}=\{f_j\}_{1\le m}$ on $V$ which has $S$ as the only pole (i.e. for every $j$, $S\cap V_j$ is the polar set of the meromorphic function $f_j$) is solvable on $S$;
\item[b)] every datum ${\rm(}\zcal{\rm)}=\{g_j\}_{1\le m}$ on $V$ which has $S$ as the only zero (i.e. for every $j$, $S\cap V_j$ is the zero set of the holomorphic function $g_j$) is solvable on $S$;
\end{itemize}
\end{defi}
About the interplay between generic curves and holomorphic or meromorphic functions the main results of Nishino are the following

\begin{itemize}
\item[1)]A generic curve $S$  of a complex surface $X$ is the zero set of a function $f$ which is holomorphic on a neighborhood $U$ of $S$.
(see \cite{Nishino}*{Proposition 7}). In other words a generic curve propagates locally.
\item[2)] Assume that a domain $D\Subset X$ contains a non-countable family $\fcal$ of connected compact complex curves $S$ such that $S\cap S'=\emptyset$ whenever $S\neq S'$. Then $\fcal$ contains at least one generic curve (see \cite{Nishino}*{Proposition 9}).
\setcounter{tempo}{\value{enumi}}
\end{itemize}
 \noindent These results globalize via the study of the normal families of compact curves.
\begin{itemize}\setcounter{enumi}{\value{tempo}}
\item[3)] Let $X$ be a weakly complete or compact surface that contains at least one generic curve. Then there exist a Riemann surface $R$ and a meromorphic map $f:X\to\R$ with compact fibers.
\item[4)] If $X$ is weakly complete and contains at least one generic curve then it contains at least one nonconstant holomorphic function with compact fibers. In particular, $X$ is holomorphically convex.
\end{itemize}
For the proof of these results, see \cite{Nishino}*{Section 5}.

\medskip

It is clear, \emph{a posteriori}, that a curve is generic in the sense of Nishino if and only if it belongs to Ueda's Class $\beta'$. Moreover, using the solvability of the two Cousin's problems along $S$, it is easy to show that a generic curve has to be flat, of infinite type and finite order, hence it has to belong to Ueda's Class $\beta'$.

\subsubsection{Ohsawa's paper and its generalization}\label{Ohsawa}
In \cite{Ohsawa} Ohsawa proved the following result (see \cite{Ohsawa}*{Proposition 1.4}.
\begin{teo}
Let $X$ be a \(connected\,\)  weakly complete  non singular complex surface such that $\ocal(X)\supsetneqq\C$. Then $X$ is holomorphically convex.
\end{teo}
This result was generalized to weakly $1$-complete complex surfaces in \cite{Tomassini-VA}.

The main tool used by Ohsawa was an observation on the topology of the level sets of a holomorphic function, which holds in every dimension, not only for surfaces (see \cite{Ohsawa}*{Theorem 1.1}).
\begin{pro}
Let $X$ be a weakly complete manifold, with a plurisubharmonic exhaustion function $\phi:X\to\R$,  and let $f\in\ocal(X)$ be a non-constant holomorphic function, then either
\begin{itemize}
\item[i{\rm)}] $f^{-1}(z)\cap\{x\in X\ :\ \phi(x)<c\}$ is empty or noncompact for all $z\in\C$ and $c\in\R$
\setcounter{tempo}{\value{enumi}}\end{itemize}
\noindent or
\begin{itemize}\setcounter{enumi}{\value{tempo}}
\item[ii{\rm)}]  $f^{-1}(z)\cap\{x\in X\ :\ \phi(x)<c\}$ is compact for all $z\in\C$ and $c\in\R$.
\end{itemize}
\end{pro}

The key fact to get the first result from this observation is that if $X$ is of dimension $2$, the fibers of $f$ are of dimension $1$, hence they are Stein if and only if they are noncompact.

\medskip

Observe that the condition $\mcal(X)\supsetneqq\C$ for weakly complete surface $X$ does not imply that $X$ is holomorphically convex (take for $X$ the surface $U$ of the Example \ref{Giuliana} below).

In higher dimension the existence of one holomorphic is not enough to grant holomorphic convexity. As a trivial example take $Z=X\times Y$ where $X$ is again the surface $U$ of the Example \ref{Giuliana} and $Y$ is a Stein curve; we think the following statement could be a suitable generalization of Ohsawa's theorem in higher dimensions.
\begin{conj*}
Let $Z$ be a weakly complete complex space of dimension $n+1$, $n\ge 1$, and $f_1,\ldots,f_n\in\ocal(Z)$ analytically independent {\rm(}i.e. $df_1\wedge\ldots\wedge df_n\neq 0${\rm)} holomorphic functions. Then $Z$ is holomorphically convex.
\end{conj*}

\section{Weakly complete surfaces}

In this section we look into the geometry of weakly complete surfaces. As we saw, in dimension two the minimal kernel carries a natural analytic structure and the presence of compact curves affects quite heavily the global geometry of the surface.

As we will see in the next pages, however, there are examples where the minimal kernel is not composed of compact curves, but of immersed complex curves with compact closure or, more precisely, of Levi-flat $3$-dimensional hypersurfaces whose Levi foliation has dense complex leaves. This kind of phenomenon does not have the same ``propagation'' property as the presence of compact curves, therefore we need some hypothesis that ensures us that we can extend this information from a level set to the whole surface. This is why we ask for the existence of a \emph{real-analytic} plurisubharmonic exhaustion function and, under such hypothesis, we prove a classification result.

This hypothesis is not always verified, as an example taken from \cite{Brunella} shows. In that case, our classification holds where the exhaustion can be taken to be real analytic. We also study the case of coronae, where our classification result carries over with minimal modifications.

\subsection{Examples and remarks}\label{examples} We present some examples of weakly complete surfaces, studying their geometry and the presence of holomorphic and meromorphic functions.

\begin{exa}\label{exa3.1}
Let $a_1 a_2$ be complex numbers such that
$$
0<\vert a_1\vert\le\vert a_2\vert<1 ,\>\>a_1^k\neq a_2^l
$$
for all $(k,l)\in\N^2\smallsetminus\{(0,0)\}$ and define $\tau$ by $\vert a_1\vert=\vert a_2\vert^\tau$; by hypothesis $\tau\notin\Q$.

Consider on $\C^2\smallsetminus\{(0,0)\}$ the equivalence relation $\sim$: $(z_1,z_2)\sim (a_1z_1,a_2z_2)$. The quotient space $\C^2\smallsetminus\{(0,0)\}/\sim$ is the Hopf manifold $\mathcal H$. Let $\pi$ denote the projection $\C^2\smallsetminus\{(0,0)\}\rightarrow\mathcal H$. The complex lines $\C_{z_1}=\{z_2=0\}$, $\C_{z_2}=\{z_1=0\}$ project into complex compact curves $C_1$, $C_2$ respectively. We note that the curves $C_1$ and $C_2$ are the only compact complex curves in $\hcal$.

Let $X=\hcal\smallsetminus C_2$. The function
$$
\Phi(z_1,z_2)=\frac{\vert z_2\vert^{2\tau}}{\vert z_1\vert^2}
$$
on $\C^2\smallsetminus\{(0,0)\}$ is $\sim$-invariant and so defines a function $\phi:X\rightarrow\R_{\{\ge 0\}}$; $\phi$ is proper and $\log\,\phi$ is pluriharmonic on $X\smallsetminus C_1$.The level sets $\{\phi=c\}$ contained in $X\smallsetminus C_1$ are the projections of the sets $\vert z_1\vert=c\vert z_2\vert^\tau$, $c>0$, and so foliated by the projections of the complex sets $\{z_1=c{\rm e}^{i\theta} z_2^\tau\}$ which are everywhere dense leaves, $\tau$ being irrational. If $f\in\ocal(X)$ then, by the maximum principle, $f$ is constant on a set $\{\phi=c\}$ which is of dimension $3$ and this implies that $f$ is constant on $X$. In the same way one shows that no strongly plurisubharmonic function exists on $X$. We also have $\mcal(\hcal)=\C$. Indeed let $f\in\mcal(\hcal)$ and $\tilde f=f\circ\pi$: $\tilde f\in \mcal(\C^2\smallsetminus\{(0,0)\})$ so it extends as meromorphic function on $\C^2$. Then $\tilde f=P/Q$ with $P,Q\in\ocal(\C^2)$. Set
\begin{equation*}
P(z_1,z_2)=\sum_{j,k=0}^\infty P_{jk}z^j_1z^k_2,\quad Q(z_1,z_2)=\sum_{\alpha,\beta=0}^\infty Q_{\alpha\beta}z^\alpha_1z^\beta_2.
\end{equation*}
Because of the $\sim$-invariance
\begin{equation*}
P(a_1z_1,a_2z_2)Q(z_1,z_2)=P(z_1,z_2)Q(a_1z_1,a_2z_2)
\end{equation*}
and from this we derive the identities
\begin{equation*}
P_{jk}Q_{\alpha\beta}a^j_1a^k_2=P_{jk}Q_{\alpha\beta}a^\alpha_1a^\beta_2
\end{equation*}
for all $j,k,\alpha,\beta\in\N$. Since, by hypothesis, $a_1^k\neq a_2^l$
for all $(k,l)\in\N^2\smallsetminus\{(0,0)\}$ we get $P_{jk}=0$ for $j+k>0$, $Q_{\alpha\beta}=0$ for $\alpha+\beta>0$ i.e. $P={\rm const}$, $Q={\rm const}$.

We can also prove that $X$ does not carry meromorphic function. Indeed, let $f\in\mcal(X)$ and $\tilde f=f\circ\pi$: $\tilde f\in \mcal(\C^2\smallsetminus \{z_1=0\})$. Then, since $\C^2\smallsetminus \{z_1=0\}$ is Stein
\begin{equation*}
\tilde f(z_1,z_2)=\sum_{j=0}^\infty \tilde f_j(z_1,z_2)z^{-j}_1
\end{equation*}
with $\tilde f_j(z_1z_2)$ entire. Again, using the $\sim$-invariance of $\tilde f$, we conclude that $\tilde f_j=0$ for $j>0$.\vskip.5truecm
\end{exa}

The curve $C_1$ is the minimum set of $\phi$ and it is the only compact curve of $X$. Since $\ocal(X)=\C$ the line bundle associated to $C_1$ is not trivial.
\begin{exa}\label{exa3.2}
With the notation of the previous example, we consider $X_1=\mathcal{H}\setminus(C_1\cup C_2)$ with plurisubharmonic exhaustion function $\alpha=(\log\phi)^2$. $X_1$ is a weakly complete surface, obviously of Grauert type. Here, however, the plurisubharmonic function $\alpha_1$ has a 3-dimensional minimum set, namely the quotient of the Levi-flat surface of $(\C^*)^2$ given by
$$H_0=\{(z_1,z_2)\in(\C^*)^2\ :\ |z_2|^\tau=|z_1|\}\;.$$
The pluriharmonic function on $X_1$ is, obviously, $\log(\phi)$, i.e. a befitting choice of the square root of $\alpha_1$.
\end{exa}
Another class of example is provided by total spaces of some complex line bundles over compact Riemann
surfaces (see also \cite{Ueda}).
\begin{exa}\label{exa3.3}
Let $M$ be a compact Riemann surface of genus $g>0$. It is well known that every topologically trivial line bundle can be represented by a flat unimodular cocycle, i.e. an element of $H^1(M, \mathbb S^1)$.

Consider a line bundle $L\to M$ with trivialization given by the open covering $\{U_j\}_{j=1}^n$ and transition functions $\{\xi_{ij}\}_{i,j}$ which represent a cocycle $\xi\in H^1(M,\mathbb S^1)$. We can define a function $\alpha:L\to \R$ by defining it on each trivialization as $\alpha_j:U_j\times\C$, $\alpha_j(x,w)=|w|^2$. As $|\xi_{ij}|=1$, these functions glue into $\alpha:L\to\R$, which is readily seen to be plurisubharmonic and exhaustive.

Now, consider $r>0$ and the section $f_1\in\Gamma(U_1, \xi)$ given by $f(x)\equiv r$ for all $x\in U_1$; taking all possible analytic continuations of $f_1$ as a section of the bundle $L$, we construct , for every chain  $\{U_{j_k}\}_{k\in\N}$ with $j_0=1$ and $U_{j_{k}}\cap U_{j_{k+1}}\neq\varnothing$, the sections
$f_k\equiv\xi_{j_{k}j_{k-1}}\xi_{j_{k-1}j_{k-2}}\cdots\xi_{j_1j_0}r\in\Gamma(U_{j_k}, \xi)\;.$
Representing $\xi$ as a multiplicative homomorphism $\psi_\xi:\pi_1(M)\to \mathbb S^1$, it is easy to see that the graphs of such sections glue into a compact complex manifold
if and only if $\psi(\pi_1(M))$ is contained in the roots of unity, i.e. if and only if $L^{\otimes n}$ is (analytically) trivial for some $n$, i.e. if and only if $\xi$ (as an element of the group $H^1(M,\mathbb S^1)$) is unipotent.

If that is not the case, the graphs of such sections glue into an imbedded, non closed, complex manifold, contained in  the Levi-flat hypersurface $\alpha^{-1}(r^2)$ and dense in it. The other leaves of the Levi foliation are obtained by the one constructed multiplying it by $e^{i\theta}$.

Finally, we have a pluriharmonic function $\chi:L\setminus M\to\R$ given by $\chi(p)=\log\alpha(p)$.

%Therefore, if $\xi$ is not unipotent, its total space gives an example of Grauert type surface for the case iii-a. As with the Hopf surface example, it is not hard to show that $L\setminus M$ is a Grauert type surface and falls in case iii-b.
\end{exa}
\begin{exa}\label{Giuliana}
Let $\Gamma\Subset\C\times\R$ be the lattice generated by $e_1=(0,0,1)$, $e_3=(1,0,0)$, $e_4=(0,1,\sqrt 2)$ and $\T$ the real
torus $\C\times\R/\Gamma$. $\T$ is foliated by complex curves which are everywhere dense. Let $\widetilde\Gamma$ the lattice
 generated by $\tilde{e}_1=(0,0,1,0)$, $\tilde{e}_2=(0,0,0,1)$, $\tilde{e}_3=(1,0,0,,0)$, $\tilde{e}_4=(0,1,\sqrt 2,0)$ and
$\widetilde\T$ the complex torus $\C^2/\widetilde\Gamma$: $\widetilde\T$ is the complexification of $\T$. The torus $\widetilde\T$
 is not algebraic \cite{Shafa}. Following \cite{Gig-Tom}*{Sections 5 and 6},  we consider the matrix $E=(E(\tilde{e}_{jk}))$
 where $E(\tilde{e}_{12})=1$,  $E(\tilde{e}_{21})=-1$, $E(\tilde{e}_{jk})=0$ if $jk\neq 12, 21$ and define the hermitian form $H$
 on $\C^2$  by $H(\zeta,\zeta')=E(i\zeta,\zeta')+iE(\zeta,\zeta')$. To $H$ corresponds a line bundle $\lsf\to\widetilde\T$ whose restriction to $\T$
 is positive. It follows that there exist a (connected) weakly complete neighborhood $U$ of $\T$ and a positive line bundle ${\widetilde\lsf}\to U$ which extends $\lsf$.
By a theorem of Hironaka (see \cite{Nakano}) $U$ embeds in some $\Pro^N$ (see \cite{Gig-Tom}*{Sections 5 and 6}). In particular $\mcal(U)\neq\C$. On the other hand $\ocal(U)=\C$ since every holomorphic function in $U$ must be constant on $\T$ whence on $U$.
\end{exa}
\begin{rem}
In the first two examples one has $\ocal(X)=\mcal(X)=\C$ while for the third $\ocal(X)=\C$ but $\mcal(X)$ contains many meromorphic functions. Moreover in the example \ref{exa3.1} $X'=X\smallsetminus C_1$ is a corona, see Section \ref{annuli} for the precise definition, for which $\ocal(X')=\C$.
\end{rem}
\subsection{Classification results}\label{class}

In the previous examples the weakly complete surfaces have an exhaustion function which is real analytic, but this is not true in general. Indeed, an example by Brunella (see \cite{Brunella}*{Theorem 1} and Section \ref{Brunella}) shows that there exist weakly complete surfaces which do not admit a  a real-analytic exhaustion function.
Without the real-analyticity as a bridge from local to global, a ``classification'' of weakly complete  surfaces in the general case seems to be very hard, so we restrict ourselves to the real-analytic case.

We consider a complex surface $X$ endowed with a real-analytic, plurisubharmonic exhaustion function $\alpha:X\to\R$. In a joint paper with Slodkowski \cite{Mon-Slo-Tom}, we classified  such complex surfaces. By passing to a desingularization, the theorem stated for a complex surface $X$, covers also the case of singular spaces.

Before getting involved in the main results of \cite{Mon-Slo-Tom}, we point out some consequences of the results we presented in the previous Section.
\begin{itemize}
\item Clearly, there are no obstruction for $X$ to contain compact negative curves.
\item No compact positive curve and no compact curve belonging to Class $\alpha$ is present on $X$. Indeed, by Suzuki (see \ref{zuki}) and Ueda (\ref{Ue}) such a curve should have strongly pseudoconcave neighborhoods and this would force $\alpha:X\to\R$ to be constant.
\item If $X$ contains a curve $C$ of  Class $\beta'$, there exist uncountably many compact complex curves near $C$ (see \ref{Ue}, and property $(\beta'_1)$). Then by Nishino results (see Section \ref{Nino}) $X$ is proper over an open holomorphic curve.
\item If $X$ contains $C$ belonging to Class $\beta''$ then there is a neighborhood of $C$ which is of Grauert type (see Section \ref{nbhcurve})
\item If $\ocal(X)\neq\C$, then, by Ohsawa's result, $X$ is either a modification of a Stein space or, if it contains a non-negative curve, is proper over an open complex curve.
\end{itemize}

\begin{question}Is it possible to have a weakly complete space that contains a compact curve belonging to Class $\gamma$? How do plurisubharmonic functions behave in the neighborhood of such curve?\end{question}

We now turn our attention to the classification result for weakly complete complex surfaces $X$, endowed with a real-analytic plurisubharmonic exhaustion function $\alpha:X\to\R$.

\begin{teo}\label{teo_Main}
Consider a non singular complex surface $X$ as above. Then, one of the following three cases occurs:
\begin{itemize}
\item[i{\rm)}] $X$ is a modification of a Stein space of dimension $2$
\item[ii{\rm)}] $X$ is proper over a {\rm(}possibly singular{\rm)} open Riemann surface
\item[iii{\rm)}] the regular level sets of $\alpha$ are compact Levi-flat surfaces foliated with dense complex leaves.
\end{itemize}
\end{teo}
A weakly complete surface $X$ which carries a smooth plurisubharmonic exhaustion function $\varphi$ whose regular level sets are Levi flat hypersurfaces, foliated by dense complex leaves (along which the Levi form of $\varphi$ degenerates) is said to be a space of \emph{Grauert type} (see \cite{Mon-Slo-Tom}, \cite{Mon-Slo-Tom1}) as their structure generalizes an example by Grauert (see, for instance \cite{Nar});

Cases i) and ii) end up being holomorphically convex, with Remmert reductions of dimension $2$ and $1$ respectively, whereas, in the third case iii), no non-constant holomorphic function exists: indeed, any holomorphic function would be constant along the complex leaves that foliate the regular levels of $\alpha$, but then it would be constant on the whole level, which is of real dimension $3$, so it would be constant on $X$.

The peculiar geometry of Grauert type surfaces does not only affect holomorphic functions, but also plurisubharmonic functions. We have the following result.

\begin{pro}
Let $X$ be a Grauert-type surface with a real-analytic plurisubharmonic exhaustion function $\alpha:X\to\R$ and let
$$M=\big\{x\in X\ :\ \alpha(x)=\min_{X}\alpha\big\}\;.$$
Then, we have two possibilities:
\begin{itemize}
\item[iii-a{\rm)}] $M$ is a compact complex curve and there exists a proper pluriharmonic function $\chi$ on $X\setminus M$ such that every pluri\-sub\-harmonic function on $X\setminus M$ is of the form $\gamma\circ\chi$
\item[iii-b{\rm)}] $M$ has real dimension $3$ and there exist a double holomorphic covering map $\pi:X^*\to X$ and a proper pluriharmonic function $\chi^*:X^*\to\R$ such that every plurisubharmonic function on $X^*$ is of the form $\gamma\circ \chi^*$.
\end{itemize}
In both cases, $\gamma$ is a convex, increasing real function.
\end{pro}

In addition, it is also true that the set $\textrm{Crt}(\alpha)$ of critical points of $\alpha$ has the same dimension of $M$.

\medskip

The two previous statements are the content of the Main Theorem in \cite{Mon-Slo-Tom}. Moreover, in \cite{Mon-Slo-Tom1}, we analyze further the geometry of Grauert-type surfaces, obtaining the following results:
\begin{itemize}
\item the level sets of the pluriharmonic function $\chi$ (or $\chi^*$) are connected, hence the function is somehow minimal,
\item any compact curve not contained in $M$ is negative in the sense of Grauert,
\item there are examples of Grauert-type surfaces which do not posses any pluriharmonic function, but their double covering does.
\end{itemize}

We do not want to enter the details of the proofs, for which we refer the reader to the papers \cites{Mon-Slo-Tom, Mon-Slo-Tom1} and the introductory note \cite{Mon-Slo-Tom0}.

Here, we only want to remark that most of the methods used are of a ``local'' nature, i.e. they work in a neighborhood of a level of the exhaustion function. The main tool in our investigation was the minimal kernel $\Sigma_X$, defined in \cite{Slo-Tom}, and its intersections with the level sets of an exhaustion function.

\medskip

As an example of the consequences of this ``local'' nature, we examine the case when all our hypotheses hold outside of a compact set.

We first observe that if $X$ carries an exhaustion function $\phi$ which is plurisubharmonic away from a compact subset $K$ then $X$ is weakly complete. Indeed, let $c$ such that $K\Subset\{\phi<c\}$ and take as a new exhaustion function $\psi:=\max(\phi,c')$ with $c'>c$.

Then we deduce the following

\begin{teo}\label{teo_buco} Let $X$ be a complex surface, $\phi:X\to\R$ an exhaustion function and $K\Subset X$ a compact set such that $\phi$ is real analytic and plurisubharmonic outside $K$. Then, one of the following cases occurs:
\begin{itemize}
\item[1{\rm)}] $X$ is a modification of a Stein space
\item[2{\rm)}] $X$ is proper over a \(possibly singular\) open Riemann surface
\item[3{\rm)}] $X\setminus K$ is of Grauert type.
\end{itemize}
\end{teo}
\begin{proof} By the above remark $X$ is weakly complete so we can repeat verbatim the proof of \cite{Mon-Slo-Tom}*{Theorem 4.4}, restricting ourselves to $X\setminus K$. If we have a sequence $c_n\to+\infty$ such that $\{\phi=c_n\}$ is contained in $X\setminus K$ and is strictly pseudoconvex, then we obtain that $X$ is a modification of a Stein space; if we have a regular level of $\phi$ in $X\setminus K$ containing a compact curve, then, by \cite{Mon-Slo-Tom}*{Theorem 4.2}, $X$ would be foliated in compact complex curves. So, by \cite{Mon-Slo-Tom}*{Corollary 4.3}, all the regular level sets of $\phi$ contained in $X\setminus K$ are Levi flat hypersurfaces foliated with dense complex leaves.
\end{proof}

One striking difference between the first two cases and the Grauert-type surfaces is the ``propagation'' of the information, when there is no real-analytic exhaustion to act as a bridge from local to global.

If there exists $c\in\R$ such that the level set $\{x\in X\ :\ \alpha(x)=c\}$ is strictly pseudoconvex then the sublevel set $\{x\in X\ :\ \alpha(x)<c\}$ is a modification of a Stein space; even more, if $\{x\in X\ :\ \alpha(x)=c\}\cap\Sigma_X=\emptyset$, then we can find a strictly pseudoconvex hypersurface, arbitrarily close to the level set of $\alpha$, therefore we can approximate the sublevel set with strictly pseudoconvex domains (which are modifications of Stein spaces), implying that also our sublevel is a modification of a Stein space.

If there exists $c\in\R$ such that the level set $\{x\in X\ :\ \alpha(x)=c\}$ contains uncountably many compact complex curves, or, equivalently, a generic compact complex curve (in the sense of Nishino, see Section \ref{Nino}), then the whole manifold $X$ is proper over an open complex curve; in this case, the information does not only extend to ``fill the hole'' in the sublevel set, but to the whole of $X$.

If there exists $c\in\R$ such that $\{x\in X\ :\ \alpha(x)=c\}$ is of Grauert type, we are sure that no generic curves are present in $X$ (otherwise $X$ would be union of compact complex curves, by Nishino); we are also sure that all the level sets $\{x\in X\ :\ \alpha(x)=c'\}$ with $c'>c$ intersect $\Sigma_X$, otherwise the corresponding sublevel would be a modification of a Stein space containing a compact Levi-flat hypersurface, which is absurd. However, we are not able to say much about the level sets with $c'<c$.

This difference is at the core of the example by Brunella (see Section \ref{Brunella}) of a weakly complete surface without real-analytic plurisubharmonic exhaustion functions; as a partial result, we note that we can exploit the existence of some particular pluriharmonic functions on a Grauert-type surface, in order to extend the information to a sublevel, given the appropriate topological condition.

By standard cohomologicaly techniques, we have the following extension result for pluriharmonic functions:
\begin{pro}\label{Hartogs_P}Suppose $W$ is a complex manifold with $H^2_c(W,\R)=0$, $U\subseteq W$ an open domain and $K\Subset U$ a compact set. Then every pluriharmonic function $\chi:U\setminus K\to\R$ has a pluriharmonic extension $\widetilde{\chi}:U\to\R$.
\end{pro}
%\begin{proof}..\end{proof}

\begin{cor}\label{cor12}Suppose $X$ is a complex surface such that $H^2(X,\R)=0$. If there exist an exhaustion function $\phi:X\to\R$ and $c\in\R$ such that
\begin{itemize}
\item[1{\rm)}] $X_c=\{\phi\leq c\}\Subset X$ is connected
\item[2{\rm)}] $X\setminus X_c$ is of Grauert type
\item[3{\rm)}] $\phi$ is real analytic and plurisubharmonic on $X\setminus X_c$
\end{itemize}
then $X$ is of Grauert type.
\end{cor}
\begin{proof} For every $\epsilon>0$ we can find $c'\in(c,c+\epsilon)$ such that $\{\phi=c'\}$ is a connected regular level, so, by \cite{Mon-Slo-Tom}*{Corollary 4.3}, we have a neighborhood $V$ of $\{\phi=c'\}$ and a pluriharmonic function $\chi:V\to\R$, which is given by $\chi=\lambda\circ\phi$ on $V$.

Therefore, there exists $c''\in(c,c')$ such that $\{c''<\phi<c'\}\subseteq V$; we set $U=\{\phi<c'\}$ and $K=\{\phi\leq c''\}$. By duality $H^2_c(X,\R)=0$,so Proposition \ref{Hartogs_P} applies giving that $\chi$ extends to a pluriharmonic function on $U$. Its level sets in $K$ are compact and Levi flat. If there is a compact complex curve in a regular level of $\chi$, $X$ is foliated in complex curves, by \cite{Mon-Slo-Tom}*{Theorem 4.2}, so, as $X\setminus K$ is of Grauert type, this is not the case. Hence, by \cite{Mon-Slo-Tom}*{Corollary 4.3}, the regular level sets of $\chi$ have dense complex leaves.
\end{proof}

This, in particular, implies that, in the situation described in Corollary \ref{cor12}, there is a global real-analytic plurisubharmonic exhaustion function.

\subsection{Coronae of dimension 2}\label{annuli}
Many of the results of \cite{Mon-Slo-Tom} are ``local'', i.e. hold in a neighborhood of a level, and the topological condition needed on the plurisubharmonic function is the properness, i.e. the fact that the level sets are compact. Let $X$ be a complex space with a smooth plurisubharmonic function $\phi:X\to(0,+\infty)$ such that
\begin{itemize}
\item $\inf_X\phi=0$, $\sup_X\phi=+\infty$
\item for every $0<\epsilon\leq m<+\infty$ the subcorona
$$X_{\epsilon, m}=\{x\in X\ :\ \epsilon<\phi(x)<m\}$$
is relatively compact in $X$.
\end{itemize}
Such a function is called an \emph{corona exhaustion} and $X$ is called a \emph{corona}. We define $\Sigma_X$ as the set of points $x\in X$ such that every smooth plurisubharmonic corona exhaustion is not strictly plurisubharmonic; it is easy to see that all the ``local'' results on the minimal kernel used or proved in \cite{Mon-Slo-Tom} extend to this setting. We would like to obtain a classification theorem for coronae admitting a real-analytic plurisubharmonic corona exhaustion; one of the ingredients that use pseudoconvexity in a global way is the result by Nishino. We show how to adapt it to the case of a corona.

\begin{lem}\label{lmm_nishino}Let $X$ be a complex surface, with a real-analytic plurisubharmonic corona exhaustion $\alpha:X\to(0,+\infty)$. Suppose $c\in\R$ is a regular value for $\alpha$, $Y\subseteq\{x\in X\ :\alpha(x)=c\}$ a connected component containing a compact complex curve $C$ and $W$ a neighborhood of $Y$ with a pluriharmonic function $\chi:W\to\R$ such that $Y=\{x\in W\ :\ \chi(x)=0\}$.

Then, $X$ is proper over an open complex curve.
\end{lem}
\begin{proof}We can assume that $C$ is connected and that $\chi$ does not have critical points in $W$. By \cite{Mon-Slo-Tom}*{Lemma 4.1}, there exist a neighborhood $V$ of $C$ in $W$ and a holomorphic function $f:V\to\C$ such that $C=\{x\in V\ :\ f(x)=0\}$; therefore, the open set $V$ is foliated in compact curves. This means that
$$\de\debar\alpha\wedge\de\alpha\wedge \debar\alpha=0$$
on $V$, hence on $X$, by real analyticity. Therefore, every regular level set of $\alpha$ is Levi flat and foliated by immersed complex curves.

Let $\{a_n\}_{n\in\N}$, $\{b_n\}_{n\in\N}$ be sequences of regular values for $\alpha$ such that $a_n$ decreases to $0$ and $b_n$ increases to $+\infty$; then the coronae
$$X_n=\{x\in X\ :\ a_n<\alpha(x)<b_n\}$$
have Levi-flat boundaries, therefore they give an exhaustion of $X$ by relatively compact pseudoconvex sets.

We now apply \cite{Nishino}*{III.5.B} to obtain that $X$ is proper over an open complex curve.
\end{proof}

In the previous Lemma, we actually proved that $X$ is a \emph{false corona}, meaning that it admits not only the function $\alpha$, but also some other real-analytic plurisubharmonic function $\beta$ which is exhaustive on $X$.

%The following theorem generalizes \cite[Theorem 4.2]{Mon-Slo-Tom}:
%\begin{teo}\label{teo33}Let $X$ be a weakly complete annulus, $W$ a domain and $\chi:W\to\R$ a pluriharmonic function. Suppose that a regular level of $\chi$ contains a compact complex curve $C$. Then there exists $\e>0$ such that the subannulus $Y=\{\phi>\e\}$ is proper over a \(possibly singular\) complex curve.
%\end{teo}
%\begin{proof}
%We may assume that $C$ is connected and $W$ is a sublevel: $W=X_{\e,c}=\{\e<\phi<c\}$. By \cite[Lemma 4.1]{Mon-Slo-Tom} there exists a neighborhood $V\Subset W$ of $C$ and a holomorphic function $f:V\to\C$ such that $C=\{f=0\}$. It follows that the family $\fcal_0=\{f=\z,\vert\z\vert<\delta\}$, for some $\delta>0$,  consists of compact complex curves.
%
%Therefore,
%Then, acting as in \cite[III.5.B]{Nishino} we extend $\fcal_0$ to a family $\fcal$ on the subannulus $Y=\{\phi>\e\}$, globally defined by a holomorphic proper map $\Phi:Y\to\C$ where $R$ is an open Riemann surface. In particular,
%$\ocal(Y)\supsetneqq\C$.
%\end{proof}
We derive the following structure theorem
\begin{teo}\label{teo34}Let $X$ be a complex surface, $\alpha:X\to\R$ a real-analytic plurisubharmonic corona exhaustion function. Then the following three cases can occur:
\begin{itemize}
\item[1{\rm)}] $X$ is an increasing union of subcoronae with one strictly pseudoconvex boundary and $\Sigma_X$ is contained in countably many level sets of $\alpha$,
\item[2{\rm)}] $X$ is proper over a complex curve, therefore a false corona,
\item[3{\rm)}] $X$ is of Grauert type.
\end{itemize}
\end{teo}
\begin{proof}
First, we assume that there exists $a\in\R$ such that $[a,+\infty)\subseteq\alpha(\Sigma_X)$; if every regular level $\{x\in X\ :\ \alpha(x)=b\}$, for $b\geq a$, contains a compact curve, then we have uncountably many compact curves in $X$ and, by \cite{Nishino}*{Proposition 9 and 7}, one of these curves has a neighborhood foliated with compact curves; therefore, reasoning as in the proof of Lemma \ref{lmm_nishino}, one regular level set is foliated and Levi flat, hence by \cite{Mon-Slo-Tom}*{Lemma 3.7} it is the level set of a pluriharmonic function. Applying Lemma \ref{lmm_nishino}, we obtain that $X$ is proper over an open (possibly singular) complex curve.

\medskip

Let us suppose that there is a regular value $c\in\R$ for $\alpha$ such that the level set $Y=\{x\in X\ :\ \alpha(x)=c\}$ does not contain any compact complex curve and $Y\cap\Sigma_X\neq\emptyset$. Then, by \cite{Mon-Slo-Tom}*{Theorem 3.6}, $\Sigma_X\supseteq Y$, $Y$ is Levi flat and $\de\debar\alpha\wedge\de\alpha\wedge\debar\alpha=0$ on $X$. Hence, every regular level is Levi flat and, by \cite{Mon-Slo-Tom}*{Lemma 3.7}, is the (regular) level set of a pluriharmonic function; by the same reasoning used in the proof of \cite{Mon-Slo-Tom}*{Corollary 3.8}, $Y$ is foliated with dense complex curves.

Therefore, if such $c$ exists, no regular level set can contain a compact curve, otherwise it would propagate as we described earlier, by Lemma \ref{lmm_nishino}; so, every regular level is foliated with dense complex curves, i.e. $X$ is of Grauert type.

\medskip

The last remaining case to consider is when there exists a sequence of regular values that do not intersect $\Sigma_X$.

We note that if $\Sigma_X$ intersects a regular level which does not contain a compact curve, then $\Sigma_X=X$; likewise, if there are uncountably many levels that contain a compact curve, then $\Sigma_X=X$. Therefore, there are at most countably many regular level intersecting $\Sigma_X$ and they all contain a compact curve; let us denote by $\{c_n\}_{n\in\N}$ the sequence of associated regular values.

By \cite{Mon-Slo-Tom}*{Proposition 3.5}, $\Sigma_X\cap\{x\in X\ :\ \alpha(x)=c_n\}$ can locally be written as a union of holomorphic discs; reasoning as in \cite{Mon-Slo-Tom}*{Theorem 3.6}, we see that there are only two cases: either $\{x\in X\ :\ \alpha(x)=c_n\}$ is Levi flat and foliated in complex curves, or $\Sigma_X\cap\{x\in X\ :\ \alpha(x)=c_n\}$ is the union of finitely many connected compact complex curves. Were the level set Levi flat, it would be the zero set of a pluriharmonic function, but as it contains at least one compact curve, this would imply $\Sigma_X=X$, which is not the case.

Therefore, $\Sigma_X\cap\{x\in X\ :\ \alpha(x)=c_n\}$ is a union of finitely many compact curves, for every $n$. Finally, if $d\in\R$ is a singular value of $\alpha$ such that $\Sigma_X\cap\{x\in X\ :\ \alpha(x)=d\}\neq\emptyset$, then for every $p$ regular point in $\{x\in X\ :\ \alpha(x)=d\}$ there exists a neighborhood $U$ such that $U\cap\Sigma_X\cap\{x\in X\ :\ \alpha(x)=d\}$ is a union of finitely many complex discs.

In conclusion, if $\Sigma_X\neq X$, then $\Sigma_X$ is a countable union of compact curves and curves immersed in the singular levels, plus, maybe, the critical points of $\alpha$; moreover, $\alpha(\Sigma_X)$ is countable and (as $\alpha$ is proper) closed.  Therefore we can find intervals $I_n$ of regular values, containing arbitrarily large real numbers, such that $\alpha(\Sigma_X)\cap I_n=\emptyset$, so we can write $X$ as a union of open domains $\Omega_n=\{x\in X\ :\ \alpha(x)<t_n\}$ with a strictly pseudoconvex boundary such that $\Sigma_X\cap \Omega_n$ is a union of compact curves, immersed curves in singular levels and critical points of $\alpha$.
\end{proof}

One could ask whether, in cases $1$ and $3$, it could be possible to ''fill the hole'' and see the corona as a subdomain of a weakly complete surface. We show that, at least in the first case, it is not possible.

We consider an example taken from Section 6 of \cite{Rossi}; we will briefly recall the construction. Let $M$ be the complex manifold diffeomorphic to $\C^2\setminus\{(0,0)\}$ and endowed with the unique complex structure such that the $2$-form
$$\phi=dz_1\wedge dz_2+\epsilon\de\overline{\de}\log(|z_1|^2+|z_2|^2)$$
is holomorphic (such a complex structure exists and is unique by a theorem of Andreotti). A function $f:M\to\C$ is holomorphic if and only if $df\wedge \phi=0$ and one can verify that the functions
$$v_{1}=\frac{z_1^2}{2}-\frac{\epsilon z_1\overline{z}_2}{|z_1|^2+|z_2|^2}\qquad v_2=\frac{z_2^2}{2}-\frac{\epsilon z_2\overline{z}_1}{|z_1|^2+|z_2|^2}\qquad v_3=\frac{z_1z_2}{2}-\frac{\epsilon z_2\overline{z}_2}{|z_1|^2+|z_2|^2}$$
are holomorphic on $M$ and satisfy $v_1^2=v_2v_3$. Hence, the map $v:M\to\C^3$ given by $v=(v_1,v_2,v_3)$ sends $M$ on the quadratic cone $K=\{(x,y,z)\in\C^3\ :\ x^2=yz\}$ minus the origin; it is easy to check that $v(z_1,z_2)=v(w_1,w_2)$ if and only if $z_1=\pm w_1$, $z_2=\pm w_2$, hence $v$ is a $2$-to-$1$ covering map.

If one considers the blow-ups $Q$ of $\C^2$ in $(0,0)$ and $Q'$ of $K$ in $(0,0,0)$, then we have a $2$-to-$1$ map from $Q\setminus Z$ to $Q'\setminus Z'$, where $Z$, $Z'$ are the respective exceptional divisors; however, the induced map $M\mapsto Q'\setminus Z'$ is not holomorphic. We consider $Q'$ as the total space of a line bundle over $\mathbb{CP}^1$, chosing coordinates $[x_1:x_2]$ on $\mathbb{CP}^1$ and coordinates $y_i$ for the fibers on $U_i=\{x_i\neq 0\}$; the transition function is then $y_2=x_1^2y_1$ on $U_1\cap U_2$. We define new coordinates
$$\tilde{y}_1=\frac{y_1}{2}-\frac{\epsilon \bar{x}_1}{1+|x_1|^2}\textrm{ on }U_1$$
$$\tilde{y}_2=\frac{y_2}{2}-\frac{\epsilon \bar{x}_2}{1+|x_2|^2}\textrm{ on }U_2$$
with transition function $\tilde{y}_2=(x_1^2+\epsilon)\tilde{y}_1$. Let $\tilde{Q}$ be the complex manifold obtained from $Q'$ with this new choice of holomorphic coordinates, then the map $M\to \tilde{Q}$ is holomorphic and $\tilde{Q}$ is Stein. Indeed, $\tilde{Q}$ can be embedded in $\C^3$ as the hypersurface $\{(w_1,w_2,w_3)\in\C^3\ :\ w_3(w_3+\epsilon)=w_1w_2\}$.

The zero section $Z'$ of $Q'$ becomes a real-analytic submanifold $A$ of $\tilde{Q}$, which is no longer complex and, actually, totally real: in a chart $U_i$, it is the graph of one of the functions $f_\pm:\C\to\C$,
$$f_\pm(z)=\pm\frac{\epsilon \bar{z}}{1+|z|^2}\;.$$
Therefore, the function $\phi(w)=\mathrm{dist}(w, A)^2$ is zero only on $A$, together with its gradient, and is strictly plurisubharmonic on $\tilde{Q}\setminus A$; the pullback $\psi$ to $M$ is a real-analytic strictly plurisubharmonic function with no critical points on $M$. The level sets $\{p\in M\ :\ \psi(p)=t\}$ for $t$ small enough are compact, as they bound a basis of neighborhoods of $(0,0)$ in $\C^2$, therefore every level set of $\psi$ is compact.

Therefore, $M$ is a corona with a real-analytic strictly plurisubharmonic corona exhaustion function and, by \cite{Rossi}, $M$ cannot be embedded as an open domain in a Stein space $\tilde{M}$ such that the complement is compact.

\begin{question}Is it true that, in case 1, $\Sigma_X$ is the union of countably many curves? Are these curve negative in the sense of Grauert?\end{question}

\begin{question}Can one produce a similar example for case 3, i.e. for Grauert-type coronae?\end{question}

\subsection{Brunella's example}\label{Brunella}
In \cite{Brunella}, Brunella gives an example of a family of weakly complete complex surfaces which do not admit any real-analytic plurisubharmonic exhaustion function; the main idea of his construction is the following.

\begin{pro}Let us consider a compact complex surface $S$ containing a compact curve $C$ which belongs to Ueda's class $\beta''$ and suppose that $S$ does not admit any holomorphic foliation tangent to $C$ and free of singularities along $C$, then $S\setminus C$ is weakly complete, but does not admit any real-analytic plurisubharmonic exhaustion function.\end{pro}
\begin{proof}By \cite{Ueda}, there exist a neighborhood $V$ of $C$ in $S$ and a function $u:V\to\R$ such that $u$ is pluriharmonic and $C=\{p\in V\ :\ u(p)=0\}$; moreover, the levels of $u$ are Levi flat and foliated with dense leaves. We consider the function $-\log u$, defined on $V\setminus C$ and we extend it to $S\setminus C$ setting it constant outside of some set $\{p\in V\ :\ \log u(p)\geq c\}$; we can do so in a smooth way.

Therefore, $S\setminus C$ is weakly complete. On the other hand, were it possible to produce a real-analytic plurisubharmonic exhaustion function, we could apply our Theorem \ref{teo_Main} and deduce that, as $S\setminus C$ contains some compact Levi-flat hypersurfaces with dense complex leaves, $S\setminus C$ has to be a Grauert type surface. The leaves of the Levi foliation give a holomorphic foliation of $S\setminus C$ which is tangent to $C$ and free of singularities along it. As this is impossible, we cannot have any real-analytic plurisubharmonic exhaustion function on $S\setminus C$.
\end{proof}

An extensive classification of foliation on surfaces was carried on in \cite{McQuillan} and we also refer to \cite{Brunella2} for the case of projective surfaces; Brunella in \cite{Brunella} considers the following explicit example. Let $C_0\subset \mathbb{CP}^2$ be a smooth elliptic curve and define $S$ as the blow-up of $\mathbb{CP}^2$ in $9$ points belonging to $C_0$ and $C\subset S$ as the strict transform of $C_0$; obviously, $(C^2)=0$. Moreover, if the $9$ points are generic, then $C\subset S$ belongs to Ueda's class $\beta''$.

From the classification of foliations mentioned above, we know that on $S$ there is no foliation tangent to $C$ with no singularities along it (see also Proposition 8 in \cite{Brunella}).

\medskip

This example by Brunella tells us that the existence of real-analytic plurisubharmonic exhaustion is not ensured by weakly completeness; however, the surface $S\setminus C$ could still be of Grauert type, i.e. union of Levi-flat hypersurfaces, where the leaves of the Levi foliations are dense, but they do not constitute a holomorphic foliation of the whole $S\smallsetminus C$.

In $S\setminus C$, we have a compact set $\Sigma\Subset S\smallsetminus C$ such that, for every neighborhood $\Omega\Subset S\smallsetminus C$ of $\Sigma$, we can produce a smooth plurisubharmonic exhaustion which is real analytic outside $\Omega$ and vanished on $\Sigma$, for every neighborhood of $\Sigma$. The complement of $\Sigma$ is the maximal pseudoflat neighborhood of $C$ in $S$.

\begin{question}What can we say about $\Sigma$? Does it have an interior? If so, what can we say about the interior? Is it weakly complete, holomorphically convex, Stein, none of the previous? Does its boundary carry any kind of analytic structure?\end{question}

\section{Minimal kernels and the structure of complex manifolds}\label{struct}

The nature of the minimal kernel of a weakly complete manifold seems to be strictly linked to its geometry and to the presence of analytic objects, at least in dimension $2$; as we recounted in the previous pages, general precise results are known only in presence of a real-analytic exhaustion. However, there are some observations that we can make for a general weakly complete surface, for example in the particular case when the minimal kernel is the whole surface.

Moreover, the notion of minimal kernel has been generalized by Slodkowski in \cite{Slo} to an arbitrary complex manifold and to other classes of plurisubharmonic function, retrieving in the process also the notion of core of a domain; Slodkowski managed to show a decomposition theorem for these generalizations of the minimal kernel, where the components are pseudoconcave and every plurisubharmonic function in the class considered is constant along each of them.

We present yet another ``kernel'' and we employ the construction and the computations in \cite{Die-For} to show that, if such kernel is empty, a bounded, smooth, pseudoconvex domain is a modification of a Stein space.

\subsection{Complex surfaces with a smooth exhaustion}\label{smoothcase}
As already observed the classification of weakly complete surfaces in the general case is very hard  and it is not even clear that the three cases of Theorem \ref{teo_Main} are the only possible.

We can nonetheless say something in some particular cases. In \cites{Mon, Mon-Slo}, the following situation was studied: $X$ is a weakly complete surface with a real-analytic plurisubharmonic exhaustion function and $\Omega\subset X$ is a domain with a smooth (in \cite{Mon}), or continuous (in \cite{Mon-Slo}), plurisubharmonic exhaustion function. Then, $\Omega$ falls into one of the three cases of Theorem \ref{teo_Main}, so it possesses a real-analytic plurisubharmonic exhaustion function.

\medskip

If $\Sigma_X$ is compact in $X$, it is quite easy and classical to prove that $X$ is a modification of a Stein space (see \cite{Mon-Slo}*{Lemma 2.1}). At the other end of the spectrum, we have the case when all the plurisubharmonic exhaustions on $X$ have everywhere degenerate Levi-form. Suppose $X$ is a complex surface endowed with a smooth plurisubharmonic exhaustion $\phi:X\to\R$, moreover assume that $\Sigma_X=X$.

We recall that, given a $k$-form $\alpha$ on a vector space $E$, the \emph{kernel} of $\alpha$ is defined as
$$\ker(\alpha)=\{v\in E\ :\ \alpha(v,v_2,\ldots, v_k)=0\ \textrm{ for all }v_2,\ldots, v_k\in E\}\;.$$
We have the following elementary lemma.

\begin{lem}\label{lmm_forme}Let $M$ be a real manifold, $\omega$ a $k$-form and $\alpha$ a $1$-form; if $\alpha\wedge\omega=0$, then for every $p\in M$ such that $\omega(p)\neq 0$ as a $k$-linear alternating form on $T_pM$, $\ker\alpha(p)\supset\ker\omega(p)$ as vector spaces in $T_pM$.\end{lem}
\begin{proof} Let $X_0\in\ker\omega(p)$ and take any $X_1,\ldots, X_k\in T_pM$; denote by $S$ the set of permutations $\sigma$ of $\{0,\ldots, n\}$ such that $\sigma(0)\neq 0$ and by $T$ the set of permutations $\tau$ of $\{1,\ldots, n\}$. Then
$$0=(\alpha\wedge\omega)(p)[X_0,X_1,\ldots, X_k]=\sum_{\tau\in T}(-1)^{|\tau|}\alpha(p)[X_0]\omega(p)[X_{\tau(1)},\ldots, X_{\tau(n)}]+$$
$$+\sum_{\sigma}\alpha(p)[X_{\sigma(0)}]\omega(p)[X_{\sigma(1)},\ldots,X_{\sigma(n)}]=$$
$$=\sum_{\tau\in T}(-1)^{|\tau|}\alpha(p)[X_0]\omega(p)[X_{\tau(1)},\ldots, X_{\tau(n)}]=n!\alpha(p)[X_0]\omega(p)[X_1,\ldots, X_n]$$
therefore, as $X_1,\ldots, X_n$ are generic and $\omega(p)\neq 0$, we need to have $\alpha(p)[X_0]=0$, so $\ker\alpha(p)\subset\ker\omega(p)$.
\end{proof}

We have an analogue of \cite{Mon-Slo-Tom}*{Lemma 5.3} for the smooth case.

\begin{lem}\label{lmm_funz}Let $W$ be a complex manifold of complex dimension $2$ and $\beta:W\to\R$ a smooth function such that
$$\de\debar\beta\wedge\de\beta=\de\debar\beta\wedge\debar\beta=0\;.$$
Suppose $|d\beta|\neq 0$ on $W$, then there exists $\mu:W\to\R$ such that
$$\de\debar\beta=\mu\de\beta\wedge\debar\beta\;.$$
\end{lem}
\begin{proof} Applying Lemma \ref{lmm_forme}, we have that
$$\ker(\de\debar\beta)\subseteq \ker(\de\beta)\cap\ker(\debar\beta)=\ker(\de\beta\wedge\debar\beta)\;
$$
whenever $\de\debar\beta\neq 0$.
Therefore, as they both are real $(1,1)$-forms, they differ by a smooth function $\mu:W\to\R$.
\end{proof}

%\begin{rem}In the hypothesis of the previous lemma
%$$0=\de\debar(\de\debar\beta)=\de\debar\mu\wedge\de\beta\wedge\debar\beta+\de\mu\wedge\debar\de\beta\wedge\debar\beta+\debar\mu\wedge\de\beta\wedge\de\debar\beta=\de\debar\mu\wedge\de\beta\wedge\debar\beta\;.$$
%By a similar reasoning, we have that $\de\debar\mu=\rho_1\de\beta\wedge\debar\beta$.
%
%\nin With the same computations we get
%$$\de\mu\wedge\de\beta\wedge\debar\beta=\debar\mu\wedge\de\beta\wedge\debar\beta=0$$
%i.e. $\de\beta\wedge\debar\beta=\rho_2\de\mu\wedge\debar\mu$. Together with the previous remark, we obtain that there exists a function $\rho:W\to\R$ such that
%$$\de\debar\mu=\rho\de\mu\wedge\debar\mu\;.$$\end{rem}
%
%We also get that $d\mu\wedge\de\beta\wedge\debar\beta=0$, so $\ker(d\mu)\supset\ker(\de\beta\wedge\debar\beta)$.

\begin{rem}If $\beta$ is plurisubharmonic,  then $\mu$ is non negative and there exists $\theta:\R\to\R$ such that $\theta\circ\mu$ is plurisubharmonic; moreover, the levels of $\beta$ are Levi flat and, by the last observation, $\mu$ is constant on the leaves of the Levi foliation, therefore also the levels of $\mu$ are Levi flat.\end{rem}

We are now in the position to state and prove our result about the structure of $X$.

\begin{teo}\label{teo_smooth}Let $X$ be a complex surface with a smooth plurisubharmonic exhaustion function $\phi$; suppose that $\Sigma_X=X$. Then, either $X$ is proper on an open Riemann surface or the connected components of the regular level sets of $\phi$ are Levi flat and foliated in dense complex curves, i.e. $X$ is of Grauert type.
In the latter case, whenever $r,s\in\R$ are such that $(r,s)$ is an interval of regular values and $W\subset\{r<\phi<s\}$ is such that $\phi\vert_W$ is proper and has connected level sets, there exists a pluriharmonic function $\chi: W\to\R$ such that $\chi=\lambda\circ\phi$.\end{teo}
\begin{proof}
As $\Sigma_X=X$, the Levi form of $\phi$ is degenerate everywhere on $X$, so every connected component of a regular level of $\phi$ is a local maximum set, hence by \cite{Slo-Tom}*{Theorem 3.9}, the forms $\de\debar\phi\wedge\de\phi$ and $\de\debar\phi\wedge\debar\phi$ vanish. We consider a connected component $W$ of $\{p\in X\ :\ r<\phi(p)<s\}$ such that all the level sets of $\phi$ contained in $W$ are regular and connected.

By Lemma \ref{lmm_funz} and the subsequent Remark, we have a non-negative function $\mu$ on $W$ such that $\de\debar\phi=\mu\de\phi\wedge\debar\phi$.

Suppose now that $d\mu\wedge d\phi$ does not vanish identically; we want to show that the generic levels of $\phi$ and $\mu$ intersect transversally.

We note that $\ker(d\mu\wedge d\phi)\supset\ker(\de\phi\wedge\debar\phi)$.

We define $F:W\to\R^2$ by $F(p)=(\phi(p),\mu(p))$; obviously, $d\mu\wedge d\phi(p)\neq 0$ if and only if $\mathrm{rk}\,DF(p)=2$. Therefore, there is an open set $U\subset W$ such that $DF$ has rank $2$, i.e. $F(U)$ is an open set in $\R^2$. By Sard's theorem, there are uncountably many regular values of $F$ in $U$. Let $c\in\R^2$ be such a regular value; then,
$$\big\{p\in W\ :\ F(p)=c\big\}=C$$
is a $2$-dimensional compact manifold (as the level sets of $\phi$ are compact), whose tangent space at every point is $\ker(\de\phi)\cap\ker(\debar\phi)$, i.e. a complex subspace of the tangent space of $X$; therefore $C$ is a compact complex curve.

This produces uncountably many compact complex curves. By Nishino, we have that $X$ is foliated by compact curves.

If $d\mu\wedge d\phi=0$, then $\mu$ is constant on the level sets of $\phi$, which are connected; therefore, we can define $m:\R\to\R$ such that $m\circ\phi=\mu$. An easy check in a coordinate patch shows us that $m$ is smooth.

We set $\chi=\lambda\circ \phi$ and we compute
$$\de\debar\chi=(\lambda''\circ\phi)\de\phi\wedge\debar\phi + (\lambda'\circ\phi)\de\debar\phi=(\lambda''\circ\phi)\de\phi\wedge\debar\phi + (\lambda'\circ\phi)\mu\de\phi\wedge\debar\phi=$$
$$=(\lambda''\circ\phi)\de\phi\wedge\debar\phi + (\lambda'\circ\phi)(m\circ\phi)\de\phi\wedge\debar\phi\;.$$

Let $t=\phi(p)$, then $\chi$ is pluriharmonic if
\begin{equation}\label{eqdiff}\lambda''(t)+\lambda'(t)m(t)=0;.\end{equation}
We can find a solution such that $\lambda'>0$. For such a choice of $\lambda$,  $\chi$ is pluriharmonic on $W$.

If there is any compact complex curve in $W$, by the argument in \cite{Mon-Slo-Tom}*{Section 4}, we propagate the curve and, by Nishino's theorem we conclude again that $X$ is proper on an open Riemann surface.

If there are no compact complex curves, we have that, on the manifold $Y=\{\chi=\chi(p)\}$ for $p\in W$, the Levi foliation is defined by the form $d^c\chi$, which is closed. Hence, the argument used in the proof of \cite{Mon-Slo-Tom}*{Corollary 3.8} implies that $Y$, which doesn't contain any compact curve, is foliated in dense complex leaves.
\end{proof}

With no effort, we obtain also a ``local'' version of the previous result, employing the full statement of \cite{Slo-Tom}*{Theorem 3.9}.

\begin{cor}Let $X$ be a complex surface with a smooth plurisubharmonic exhaustion function $\phi$. Suppose that $r\in\R$ is such that there exists a connected component $Y$ of $\{p\in X\ :\ \phi(p)=r\}$ which is relatively open in the level set and is contained in $\Sigma_X$. Then, either $X$ is proper on an open Riemann surface or there is $s<r$ such that for every regular value $t\in (s,r)$, the corresponding level set is Levi flat and foliated with dense complex curves.\end{cor}

Moreover, recalling the homological condition used in Section 3, we have the following.

\begin{cor}Suppose $X$ is a complex surface of Grauert type, endowed with a smooth plurisubharmonic exhaustion $\phi$, such that $H^2(X,\R)=0$. Suppose that there exists $c\in\R$ such that the level sets of $\phi$ contained in $X\setminus\{\phi\leq c\}$  are connected and contain at least one regular point of $\phi$. Then there exists a real-analytic plurisubharmonic exhaustion.\end{cor}
\begin{proof} We note that the regular values of $\phi$ are an open set in $\R$, because the set of critical points is closed and the function $\phi$ is proper, hence closed. So, by Sard's theorem, the set of critical values has measure $0$, so we have a sequence of regular values growing to $+\infty$; moreover, the closure of the regular values is the image of $\phi$.

For any interval $I=(a,b)$, constituted only of regular values, with $a>c$, we consider the set $W=\{a<\phi(p)<b\}$. Since $X$ is of Grauert type, $\Sigma_X=X$ and, following the proof of Theorem \ref{teo_smooth}, we have a pluriharmonic function $\chi:W\to\R$, which is obviously constant on (the connected components of) the level sets of $\phi$. Moreover, such a function $\chi$ is obtained via a function $m:I\to\R$; the construction of the function $m$ is local and works around every regular point, so we can extend it to the whole of $X\setminus\{\phi\leq c\}$ by continuity. Therefore, we obtain a function $\lambda$ satisfying \eqref{eqdiff} and, subsequently, a pluriharmonic function $\chi:X\setminus\{\phi\leq c\}\to\R$.

As $H^2(X,\R)=0$, by Proposition \ref{Hartogs_P}, we can extend $\chi$ as a pluriharmonic function on $X$ and we obtain the desired exhaustion.
\end{proof}

\subsection{The singular locus of an admissible class}

In a series of papers, Harz, Shcherbina and the second author introduced a concept analogous to the minimal kernel, dealing with bounded plurisubharmonic functions, see \cites{HST1, HST2, HST3}. The \emph{core} of a complex manifold $X$ is defined as the set $\mathbf{c}(X)$ of points where every \emph{bounded} plurisubharmonic function has a degenerate Levi form; as for the minimal kernel, such definition can be given in any regularity class, from continuous to smooth or real analytic.

The presence of the core of a domain in a complex manifold is obviously linked to the existence of strictly plurisubharmonic defining functions; as for the minimal kernel, in dimension $2$ the nature of the core is better understood and linked to the presence of some analytical objects, namely complex curves in the intersections of the core with level sets of a minimal function, where all the bounded plurisubharmonic functions are constant. For further details we refer the interested reader to the papers cited before or to the electronic preprint \cite{HST0}.

The existence of such objects is meaningful, as they, in some sense, explain the failure at being strictly plurisubharmonic because they force such functions to be constant; this point of view has been examined and studied in detail by Slodkowski in \cite{Slo}, where he generalizes the constructions of the minimal kernel and the core to what he calls singular locus of an admissible class of functions.

\begin{defi}Let $X$ be a complex manifold and let $\mathcal{U}$ be the set of all the open sets $U\subset X$ such that at least one between $U$ and $X\setminus U$ is relatively compact. An admissible class $\mathcal{F}$ is the datum, for every $U\in\mathcal{U}$, of a set $\mathcal{F}(U)$ of continuous plurisubharmonic functions on $U$ with the following properties:
\begin{enumerate}
\item if $\phi\in\mathcal{F}(U)$ and $W\in\mathcal{U}$, $W\subset U$, then $\phi\vert_W\in\mathcal{F}(W)$
\item if $\phi:X\to\R$, $U_1,\ldots, U_n\in\mathcal{U}$ form a covering of $X$ and $\phi\vert_{U_i}\in\mathcal{F}(U_i)$ for $i=1,\ldots, n$, then $\phi\in\mathcal{F}(X)$
\item $\mathcal{F}(U)$ is a convex cone and contains every bounded smooth plurisubharmonic function on $U$
\item if $\phi_1,\ldots, \phi_n\in\mathcal{F}(U)$ and $v:\R^n\to\R$ is smooth, convex, of at most linear growth and with non-negative partial derivatives, then $\phi=v(\phi_1,\ldots,\phi_n)$ belongs to $\mathcal{F}(U)$
\item if $\phi\in\mathcal{F}(U)$ is strongly plurisubharmonic and if $\rho\in\mathcal{C}^\infty(U)$ with $\overline{\mathrm{supp}\;\rho}\subset U$, then there is $t>0$ such that $\phi+t\rho\in\mathcal{F}(U)$
\item for every sequence $\phi_n\in\mathcal{F}(X)$, there are positive numbers $\epsilon_n$ such that $\sum_n\epsilon_n\phi_n$ converges uniformly on compact subsets of $X$ to a function in $\mathcal{F}(X)$.
\end{enumerate}
\end{defi}

Examples of admissible classes are the class of plurisubharmonic exhaustion function (of a given regularity), the class of bounded plurisubharmonic functions (of a given regularity), the class of all plurisubharmonic functions (of a given regularity); other examples can be obtained by adding any kind of growth condition.

\begin{defi}Let $\mathcal{F}$ be an admissible class, then the singular locus $\Sigma^{\mathcal{F}}$ of $\mathcal{F}$ is defined as the set of points of $X$ where no function of $\mathcal{F}(X)$ is strongly plurisubharmonic. A function $\phi\in\mathcal{F}(X)$ is called $\mathcal{F}$-minimal if it is strongly plurisubharmonic exactly on $X\setminus\Sigma^{\mathcal{F}}$.\end{defi}

As it is proved in \cite{Slo}*{Section 4}, the singular locus, if non empty, is pseudoconcave, since its intersection with a ball in $X$ is a local maximum set (see \cite{Slo}*{Proposition 4.4 and Lemma 4.5}). Moreover, $\Sigma^{\mathcal{F}}$ can be decomposed in pseudoconcave parts where all the functions in $\mathcal{F}(X)$ are constant \cite{Slo}*{Theorem 4.7}. For the minimal kernel, we can also prove that such pseudoconcave parts are compact (see \cite{Slo}*{Theorem 5.2}).

\medskip

We give the following definition (see \cite{Slo}*{Section 5}.
\begin{defi}Let $X$ be a complex manifold, $k\in\N\cup\{\infty\}$ and consider the admissible class $\mathcal{F}$ such that, for $U\in\mathcal{U}$,  $\mathcal{F}(U)$ is the set of all $\mathcal{C}^k$ plurisubharmonic functions on $U$. We define the \emph{minimal kernel} of $X$ as the singular locus of $\mathcal{F}$ and we denote it by $\Sigma_X$.\end{defi}

It is easy to see that, if $X$ admits a $\mathcal{C}^k$ plurisubharmonic exhaustion function, then $\Sigma_X$ coincides with the previously defined minimal kernel of a weakly complete space. The decomposition in pseudoconvex parts still holds, however they may not be compact anymore.

\subsection{A Levi problem}\label{levipb}

Let $M$ be a complex manifold.

\begin{defi} For $U\subseteq M$ open domain with smooth boundary, we define the \emph{boundary kernel} of $U$ as
$$
{\rm b}\Sigma_U=\overline{\Sigma}_U\cap {\rm b}U\;,
$$
where the closure is taken with respect to the topology of $M$.
\end{defi}

\begin{defi} For $Y\subseteq M$ closed, let $\mathcal{U}(Y)$ be a basis of open neighborhoods of $Y$ in $M$. We define the \emph{psh kernel} of $Y$ as
$$\Sigma_{Y}=\bigcap_{U\in\,\mathcal{U}(Y)}\Sigma_U\;.$$
\end{defi}

The definition of some kind of minimal functions for the boundary kernel is not obvious and we do not know any results in this direction.
\begin{question}Is there any kind of minimal function for the boundary kernel of a smoothly bounded domain?\end{question}
In the case of the psh kernel of a closed set $Y$, a reasonable definition would be a plurisubharmonic function defined in a neighborhood $Y$ which is strongly plurisubharmonic exactly outside of $\Sigma_Y$; this is equivalet to ask that there exists a neighborhood $U$ of $Y$ such that $\Sigma_U=\Sigma_Y$. If $\Sigma_Y=\emptyset$, the existence of such a function follows from the definition of the psh kernel. In general, we do not know if such a function exists.

Another interesting question is the following.
\begin{question}Are there any relations between ${\rm b}\Sigma_U$ and $\Sigma_{{\rm b}U}$?\end{question}
In general, they do not coincide, as the following example shows.

\begin{exa}\label{ex_blowup}
Let $M$ be the blow-up of $\C^2$ at the origin $\{z=w=0\}$, $\pi:M\to\C^2$ the reduction map, $E=\pi^{-1}(0,0)$ the exceptional divisor. Let $U=\{|z-1|^2+|w|^2<1\}$ and $\Omega=\pi^{-1}(U)$. Then, $U$ is Stein, so $\Omega$, which is biholomorphic to $U$, is Stein and $\phi:=\psi\circ\pi$, with $\psi(z,w)=|z-1|^2+|w|^2-1$, is a defining function for $\Omega$. It follows that ${\rm b}\Sigma_\Omega=\emptyset$. On the other hand ${\rm b}\Omega$ contains $E$, where every plurisubharmonic function has a degenerate Levi form, hence $E\subseteq \Sigma_{{\rm b}\Omega}$.\end{exa}

We do not know if the inclusion ${\rm b}\Sigma_U\subseteq\Sigma_{{\rm b}U}$ holds in general, however, if we consider the kernels $\Sigma^0$ given by continuous plurisubharmonic functions, we have the following result.

\begin{lem} Suppose $U$ is relatively compact in $M$. If there exists a function $u$ defined in a neighborhood of $\overline{U}$ and plurisubharmonic there such that $U=\{u<0\}$, then ${\rm b}\Sigma^0_U\subseteq\Sigma^0_{{\rm b}U}$.\end{lem}
\begin{proof} Let us consider the open set $V$, neighborhood of $bU$ in $M$ and $\psi$ continuous and plurisubharmonic on $U$. We can suppose that $\min\limits_{{\rm b}U}\psi\geq 1$. Let $c<0$ be such that $L=\{c\leq u\leq 0\}$ is compact in $U$ and denote by $S$ the level set $\{u=c\}$. .

Consider a function $\theta:(-\infty,0]\to(-\infty,0]$ satisfying
\begin{enumerate}
\item $\theta(x)<0$ if $x<0$ and $\theta(0)=0$
\item $\theta\circ u$ is continuous and plurisubharmonic on $\mathrm{int}(L)$
\item if $x\in \mathrm{int}(L) $ is a point of strong plurisubharmonicity for $u$, then it is also a point of strong plurisubharmonicity for $\tilde{u}=\theta\circ u$
\item $\theta(c)<c-\max_{L}\psi$.\end{enumerate}
and let $\tilde{u}=\theta\circ u$. Then
$$(\psi+\tilde{u})(x)<u(x)\quad \forall x\in S$$
$$(\psi+\tilde{u})(x)\geq 1>u(x)\quad \forall x\in bU\;.$$
The function
$$v=\max\{u,\psi+\tilde{u}, c\}$$
is continuous and plurisubharmonic on $U$.

Now, if $x\not\in \Sigma^0_{{\rm b}U}$, then we can find a $\psi$ which is strongly plurisubharmonic in a neighborhood of $x$, which implies that $v$ is strongly plurisubharmonic in the intersection of a neighborhood of $x\in {\rm b}U$ with $U$ , so $x\not\in {\rm b}\Sigma^0_{\Omega}$.
\end{proof}

The boundary kernel is clearly linked only with the geometry of the domain $U$ near its boundary, whereas the psh kernel of the boundary ${\rm b}U$ also takes into account what happens \emph{on the boundary}, as it is clearly outlined in Example \ref{ex_blowup}. This leads us to conjecture that the conclusion of the previous Lemma should hold in greater generality.

Considering the greater quantity of information encoded in the psh kernel of the boundary, it is reasonable to think that its presence or absence have a strong role in determining the geometry of the domain. Indeed, we have the following result.

\begin{teo}\label{teo_levi}Let $\Omega\Subset M$ be a smoothly bounded, pseudoconvex domain; if $\Sigma_{{\rm b}\Omega}=\emptyset$, then $\Omega$ is a modification of a Stein space.\end{teo}
\begin{proof}
If $\Sigma_{{\rm b}\Omega}$ is empty, then there exists a strongly pseudoconvex function $\psi$ defined in a neighborhood of ${\rm b}\Omega$, so, by \cite{Rich}, $\psi$ can be taken to be $\ccal^\infty$. Since $\Omega$ is pseudoconvex, then there is a $\ccal^\infty$ function $\rho:M\to\R$ such that
\begin{enumerate}
\item $\Omega=\{x\in M\ :\ \rho(x)<0\}$
\item $d\rho\neq 0$ on $b\Omega=\{\rho=0\}$
\item $i\de\debar\rho(x)$ is positive semidefinite on $T^{1,0}_x{\rm b}\Omega\oplus T^{0,1}_x{\rm b}\Omega$ for every $x\in {\rm b}\Omega$.
\end{enumerate}
We follow closely the approach and the computations of \cite{Die-For}*{Theorem 1}.

We define $u=-(-\rho e^{-C\psi})^\alpha$, where $C$, $\alpha$ are positive real constants to be determined later. Up to shrinking $U$, we can suppose that $\psi$ is bounded, so we can also assume $\psi\geq 0$.

In order to make the computations easier, we take an hermitian metric $\omega$ on $M$, such that $i\de\debar\psi\geq\omega$ in a neighborhood of ${\rm b}\Omega$.

Following \cite{Die-For}, we calculate
$$\debar u=\alpha(-\rho e^{-C\psi})^{\alpha-1}( e^{-C\psi}\debar\rho-C\rho e^{-C\psi}\debar\psi)$$
and
$$\de\debar u=-\alpha(\alpha-1)(-\rho e^{-C\psi})^{\alpha-2}( e^{-C\psi}\de\rho-C\rho e^{-C\psi}\de\psi)\wedge$$
$$\wedge( e^{-C\psi}\debar\rho-C\rho e^{-C\psi}\debar\psi)+$$
$$+\alpha(-\rho e^{-C\psi})^{\alpha-1}( e^{-C\psi}\de\debar\rho-C e^{-C\psi}\de\psi\wedge\debar\rho-$$
$$-C\rho e^{-C\psi}\de\debar\psi-C e^{-C\psi}\de\rho\wedge\debar\psi+C^2\rho e^{-C\psi}\de\psi\wedge\debar\psi)\;.$$
We factor out the term $\alpha(-\rho e^{-C\psi})^{\alpha-2} e^{-2C\psi}$, which is always positive in $U\cap\Omega$, and we obtain
$$-(\alpha-1)\de\rho\wedge\debar\rho + 2\alpha C\rho\Re(\de\rho\wedge\de\psi)-\alpha C^2\rho^2\de\psi\wedge\debar\psi-\rho\de\debar\rho+C\rho^2\de\debar\psi=$$
$$=C\rho^2(\de\debar\psi-C\alpha\de\psi\wedge\debar\psi)-\rho(\de\debar\rho-2\alpha C\Re(\de\rho\wedge\debar\psi))+(1-\alpha)\de\rho\wedge\debar\rho\;.$$
We note that
$$|2\alpha \rho\Re(\de\rho\wedge\debar\psi)|\leq (\alpha^2|\de\rho|^2+C^2\rho^2|\de\psi|^2)\omega\;.$$
The pseudoconvexity of ${\rm b}\Omega$ can be stated by saying that $i\de\debar\rho$ is positive semidefinite on $\ker \de\rho$ in $T^{1,0}_zM$ for $z\in b\Omega$; therefore we can find $C_1>0$ such that
$$i\de\debar\rho\geq -C_1|\de\rho|\omega$$
on $b\Omega$. Now, if $z$ is close enough to $b\Omega$, we can find a point $w\in b\Omega$ such that
$$i\de\debar\rho(z)\geq i\de\debar\rho(w) - C_2|\rho(z)|\omega$$
as $\rho(z)$ is comparable with the distance (with respect to $\omega$) from $z$ to $w$. Therefore, we have a constant $D>0$ such that
$$i\de\debar\rho(z)\geq - D(|\de\rho|+|\rho|)\omega\;.$$
Therefore
$$C\rho^2(\de\debar\psi-C\alpha\de\psi\wedge\debar\psi)-\rho(\de\debar\rho-2\alpha C\Re(\de\rho\wedge\debar\psi))+(1-\alpha)\de\rho\wedge\debar\rho\geq$$
$$\geq C\rho^2(\de\debar\psi-C(1+\alpha)|\de\psi|^2\omega)+(1-\alpha-\alpha^2)|\de\rho|^2+D\rho|\de\rho|\omega-D\rho^2\omega$$
and, as $D|\rho||\de\rho|\leq \delta D (|\de\rho|^2+\delta^{-1}\rho^2))$, we get the lower bound
$$C\rho^2(\de\debar\psi-C(1+\alpha)|\de\psi|^2\omega - C^{-1}D(1+\delta^{-1})\omega)+(1-\alpha-\alpha^2-D\delta)|\de\rho|^2\omega\;.$$
As $0\leq \psi\leq m$, by replacing $\psi$ with $\psi^2/M$, we have that $i\de\debar\psi\geq iM'\de\psi\wedge\debar\psi$, hence
$$\de\debar\psi-C(1+\alpha)|\de\psi|^2\omega - \frac{D}{C}(1+\delta^{-1})\omega\geq (|\de\psi|^2(M'-C(1+\alpha))- \frac{D}{C}(1+\delta^{-1}))\omega$$
and $|\de\psi|^2$ can be supposed to be greater than a positive constant $K$ along ${\rm b}\Omega$. We can choose $M'$, $C$, $\alpha$ and $\delta$ such that $C\sim \delta^{-1}\sim\alpha^{-2}$ and $M'>2C+2DK$ and, finally, $\alpha$ small enough so that $(1-\alpha-\alpha^2-D\delta)>0$.

Hence $u$ is strictly plurisubharmonic in $U\cap \Omega$ and $\{u=0\}=b\Omega$. Therefore, $\{u=c\}$, for $c<0$ and small enough in absolute value, is compact in $U\cap\Omega$; so, we define an exhaustion function $\tilde{u}=\theta\circ\max\{u, c\}$, where $\theta$ is a convex, increasing real function. Then, $\Omega$ is a modification of a Stein space.
\end{proof}

\begin{rem}We note that $\Sigma_{{\rm b}\Omega}=\emptyset$ is not a necessary condition for $\Omega$ to be a modification of a Stein space, as Example \ref{ex_blowup} shows.
\end{rem}

The previous remark leads us to formulate a conjecture.

\begin{conj*}$\Omega$ is a modification of a Stein space if and only if ${\rm b}\Sigma_{\Omega}=\emptyset$.\end{conj*}

\begin{rem}If $\Omega$ is weakly complete and ${\rm b}\Sigma_\Omega=\emptyset$, then $\Omega$ is a modification of a Stein space, by \cite{Mon-Slo}.
\end{rem}

	\providecommand{\bysame}{\leavevmode\hbox to3em{\hrulefill}\thinspace}
\renewcommand{\MR}[1]{}
\providecommand{\MRhref}[2]{%
  \href{http://www.ams.org/mathscinet-getitem?mr=#1}{#2}
}
\providecommand{\href}[2]{#2}

\end{document}